\documentclass[11pt]{article}

\usepackage{latexsym,amsfonts,amssymb}
\usepackage{amsmath,amsthm,amscd}
\usepackage{hyperref,psfrag}

\usepackage[dvips]{epsfig}

\usepackage{graphicx}

\usepackage{comments}

\usepackage{a4wide}

\usepackage{tikz}
\usetikzlibrary{decorations.markings}
\usetikzlibrary{decorations.pathreplacing}
\usetikzlibrary{decorations.pathmorphing,shapes}
\usepackage[outline]{contour}
\contourlength{1.2pt}

\tikzset{->-/.style={decoration={
  markings,
  mark=at position #1 with {\arrow{>}}},postaction={decorate}}}

\tikzset{middlearrow/.style={
        decoration={markings,
            mark= at position 0.5 with {\arrow{#1}} ,
        },
        postaction={decorate}
    }
}

\newcommand{\hackcenter}[1]{
 \xy (0,0)*{#1}; \endxy}

\psfrag{X1}{$\sigma=(1,2,4)$} 
\psfrag{X2}{$\ell(\sigma) = 4$} 
\psfrag{X3}{$x_1$}
\psfrag{X4}{$x_2$}
\psfrag{X5}{$y_1$}
\psfrag{X6}{$y_2$}
\psfrag{X7}{$x$}
\psfrag{R11}{$y_1$}
\psfrag{R12}{$y_2$}
\psfrag{R10}{$x$} 
\psfrag{X8}{$k$}
\psfrag{X9}{$n-k$}
\psfrag{X10}{$m$}
\psfrag{X11}{\mbox{\scriptsize{$\small{k-1}$}}}
\psfrag{X12}{\mbox{\scriptsize{$m-1$}}}
\psfrag{X13}{\mbox{\scriptsize{$n-m-k+1$}}}

\normalfont\upshape

\newcommand{\sign}{\text{sign}}
\newcommand{\row}{\text{row}}
\newcommand{\col}{\text{col}}
\newcommand{\cont}{\text{cont}}

\usepackage{fancyheadings}
\usepackage[all]{xy}
\SelectTips{cm}{}

\theoremstyle{definition}
\newtheorem{thm}{Theorem}[section]
\newtheorem{cor}[thm]{Corollary}

\newtheorem{lem}[thm]{Lemma}

\newtheorem{prop}[thm]{Proposition}

\newtheorem{example}[thm]{Example}

\numberwithin{equation}{section}

\usepackage{bbm}
\def\C{{\mathbbm C}}
\def\N{{\mathbbm N}}

\def\Z{{\mathbbm Z}}
\def\Q{{\mathbbm Q}}

\newcommand{\Hom}{{\rm Hom}}

\renewcommand{\to}{\rightarrow}

\def\det{\mathop{\rm det}}

\def\Ind{{\mathrm{Ind}}}
\def\lra{{\longrightarrow}}

\def\mc{\mathcal}

\def\shuffle{\,\raise 1pt\hbox{$\scriptscriptstyle\cup{\mskip
               -4mu}\cup$}\,}

\newcommand{\refequal}[1]{\xy {\ar@{=}^{#1}
(-1,0)*{};(1,0)*{}};
\endxy}

\usepackage[vcentermath]{youngtab}

\title{The Hopf algebra of odd symmetric functions}
      \author{
      Alexander P. Ellis and Mikhail Khovanov}

\date{September 18, 2013}

\newcommand{\sym}{\mathrm{\Lambda}}

\newcommand{\oplusop}[1]{{\mathop{\oplus}\limits_{#1}}}

\begin{document}
%

\maketitle

\begin{abstract}
We consider a $q$-analogue of the standard bilinear form on the commutative ring of symmetric functions.  The $q=-1$ case leads to a $\Z$-graded Hopf superalgebra which we call the algebra of odd symmetric functions.  In the odd setting we describe counterparts of the elementary and complete symmetric functions, power sums, Schur functions, and combinatorial interpretations of associated change of basis relations.
\end{abstract}

\setcounter{tocdepth}{2} \tableofcontents

\newpage

%
\section{Introduction}
%

\subsection{Symmetric groups and categorification}

The ring $\sym$ of symmetric functions plays a fundamental role in several 
areas of mathematics. 
It decategorifies the representation theory of the symmetric 
groups, for it can be identified with the direct sum of Grothendieck groups of 
group rings of the symmetric groups: 
$$ \sym \ \cong \ \oplusop{n\ge 0} K_0(\C[S_n]),$$ 
where any characteristic $0$ field can be used instead of $\C$. Multiplication 
and comultiplication in $\sym$ come from induction and restriction 
functors for inclusions $S_n\times S_m\subset S_{n+m}$, and Schur functions 
are the images of simple $S_n$-modules in the Grothendieck 
group under this isomorphism.  
The elementary and complete symmetric functions $e_n$ and $h_n$ are the 
images of the sign and trivial representations of $S_n$, respectively. 

The ring $\sym(n)$ of symmetric functions in $n$ variables is 
the quotient of $\sym$ by the ideal  generated by $e_m,$ 
over all $m>n$; it is  naturally isomorphic to the representation ring of 
polynomial representations of $GL(n)$, with Schur functions $s_{\lambda}$ 
for partitions $\lambda$ with at most $n$ rows given by the symbols of the corresponding 
irreducible representations of $GL(n)$.  Good accounts of the above in the literature can be found in \cite{Bump,Fulton,Mac,Stanley2}.

These structures are deep and serve as a foundation as well as a model 
example for many further developments in representation theory. 
One such development starts with the nilHecke ring $NH_n$, the 
ring of endomorphisms of $\Z[x_1, \dots, x_n]$ generated by 
the divided difference operators $\partial_i$ and the operators of multiplication 
by $x_i$. This ring is related to the geometry of flag varieties; see \cite{Kum} 
and references there. 
More recently, $NH_n$ appeared in the categorification of quantum $\mathfrak{sl}_2$ 
\cite{Lau1,Rou2}; cyclotomic quotients of this ring categorify weight spaces of 
irreducible representations of quantum $\mathfrak{sl}_2$. The ring $NH_n$ 
admits a graphical intepretation, with its elements described by diagrams on $n$ strands. 
The generator $x_i$ is presented by $n$ vertical strands with a dot on the $i$-th 
strand, the generator $\partial_i$ by the intersection of the $i$-th  and $(i+1)$-st 
strands. See \cite{KL1,Lau1} for some uses of this diagrammatic representation.

The subring of $NH_n$ generated by the divided difference operators 
$\partial_i$ is known as the nilCoxeter ring; see \cite{kho2} and references therein. 
An odd counterpart of the nilCoxeter ring, the LOT (Lipshitz-Ozsv\'{a}th-Thurston) ring, 
recently appeared in the 
bordered Heegaard Floer homology \cite{LOT} and should play a role in 
the categorification of quantum superalgebras \cite{KhoGL12}. In this odd version 
far away crossings ($\partial_i$ and $\partial_j$ for $|i-j|>1$) 
anticommute rather than commute, with the other defining relations 
remaining the same ($\partial_i^2=0,$ $\partial_i \partial_{i+1}\partial_i = 
\partial_{i+1}\partial_i \partial_{i+1}$). It is natural to add dot generators 
to the LOT ring, making all far away generators (dots and crossings) anticommute, 
and suitably modifying the other defining relations for the nilHecke algebra. 
The resulting ``odd nilHecke'' algebra $ONH_n$ on $n$-strands shares 
many similarities with the nilHecke algebra $NH_n$; for instance, it 
acts on the space of polynomials in $n$ anticommuting variables via 
multiplication by these variables and by odd analogues of the divided difference 
operators.  The odd nilHecke algebra and its action on skew-symmetric polynomials
appear in the very recent work of Kang, Kashiwara, and Tsuchioka \cite{KKT}, 
where it is used, in particular, as a building block for super-analogues of the 
quiver Hecke algebras.  In the forthcoming 
paper \cite{EKL} we will develop the odd counterpart of the diagrammatical calculus 
\cite{KLMS} for $NH_n$.

The nilHecke algebra $NH_n$ is isomorphic to the matrix algebra 
of size $n!\times n!$ with coefficients in the ring of symmetric functions 
in $n$ variables, and the combinatorics of symmetric functions can be rethought 
from the viewpoint of $NH_n$; see \cite{KLMS,Man} for instance. 
A straightforward argument shows that the odd nilHecke algebra is isomorphic 
to the matrix algebra of size $n!\times n!$ with coefficients in the ring $\sym^{-1}(n)$ 
with generators $e_1, \dots, e_n$ and defining relations
\begin{eqnarray*} 
  e_i e_j  & = & e_j e_i \ \ \  \mbox{if} \ \ \  i+j \ \ \ \mbox{is even},  \\
  e_i e_j + (-1)^i e_j e_i  & = &  e_{j-1}e_{i+1} + (-1)^i e_{i+1} e_{j-1} \ \ \ 
 \mbox{if} \ \ \  i+j \ \ \ \mbox{is odd}. 
\end{eqnarray*} 
We would like to think of $\sym^{-1}(n)$ as the odd counterpart of the 
ring $\sym(n)$ of $n$-variable symmetric functions and 
will pursue this approach in \cite{EKL}, defining the 
odd Schur function basis of $\sym^{-1}(n)$ via its embedding in $ONH_n$, 
and showing that cyclotomic quotients of $ONH_n$ are Morita equivalent to 
suitable quotients of $\sym^{-1}(n)$ which should be odd counterparts 
of the cohomology rings of complex Grassmannians. The ring $\sym^{-1}(n)$ 
can be thought of as an odd counterpart of the cohomology ring of 
$Gr(n,\infty)$, the Grassmannian of complex $n$-planes in $\C^{\infty}$. 
We prefer to use ``odd'' rather than ``super'' here, since 
$\sym^{-1}(n)$ is not isomorphic to the cohomology ring of any 
super topological space and hints at genuinely quantum geometry.  

\vspace{0.07in}

\subsection{Outline of this paper}

By analogy with the even case, if we send $n$ to infinity, the resulting 
limit algebra $\sym^{-1}$ should be a Hopf superalgebra, the odd analogue 
of the algebra $\sym^1$ of symmetric functions (which was called $\sym$ above). In the present paper we 
develop an approach to $\sym^{-1}$ that bypasses odd nilHecke algebras. 
Fix a scalar $q$.  We first define a $q$-Hopf algebra $\sym'$ on generators $h_1, h_2, \dots \  .$ 
A $q$-Hopf algebra is a Hopf algebra in the category of graded vector spaces 
with the braiding given by $q$ to the power the product of the degrees
(in the terminology of \cite{AguiarMahajan}). The $q$-Hopf algebra $\sym'$ appears in 
Section 17.3.4 of \cite{AguiarMahajan} and has a 
natural bilinear form, which is nondegenerate over $\Q(q)$. The bilinear form degenerates 
for special values of $q$, and we can form the quotient of $\sym'$ by 
the kernel of the bilinear form for any such special value.  We denote this quotient by $\sym^q$ or just $\sym$ when $q$ is understood from the context (usually $q=-1$).  The case $q=1$ results 
in the familiar Hopf algebra of symmetric functions. Here we study the 
next case in simplicity, that of $q=-1$. The resulting quotient $\sym^{-1}$ 
is a Hopf superalgebra which is  neither cocommutative nor commutative 
as a superalgebra.  Subsections \ref{subsec-q-vector-spaces} and \ref{subsec-e-h} are devoted to establishing this basic language and the bialgebra structure on $\sym^{-1}$, as well as setting up a graphical interpretation of the bilinear form.

The Hopf algebra $\sym^1$ has an automorphism $\omega$ which swaps $e_n$ and $h_n$.  Since $\sym^1$ has the extremely rigid structure of a positive self-adjoint Hopf algebra \cite{ZelFinite}, this automorphism is uniquely determined by the fact that it preserves the bilinear form, preserves the set of positive elements, and switches $e_2$ and $h_2$.  On the categorified level, applying $\omega$ amounts to taking the tensor product with the sign representation.  The case of $\sym^{-1}$ is more complicated, due in part to the lack of commutativity and cocommutativity and the lack of a positivity structure.  In Subsection \ref{subsec-automorphisms}, we study several (anti-)automorphisms, some involutory, each of which bears some of the properties of $\omega$.  We obtain the antipode $S$ from these, completing the structure of a Hopf superalgebra on $\sym^{-1}$.

The $q$-Hopf algebra $\sym'$ is isomorphic to the graded dual of the $q$-Hopf algebra of quantum quasi-symmetric functions introduced by Thibon and Ung \cite{ThibonUng}.  In Subsection \ref{subsec-comparison-qsymq}, we explain this isomorphism.

The classical monomial sums in $\sym^1$ derive much of their importance from the fact that they form a dual basis to the basis of complete symmetric functions.  In Subsection \ref{subsec-dual-bases}, we introduce the dual bases to the bases of odd complete and elementary symmetric functions, which we call the odd monomial and forgotten symmetric functions.  Signed analogues of the classical combinatorial relations between the complete, elementary, monomial, and forgotten bases are derived as well.

Any Hopf algebra $H$ has an associated Lie algebra of primitives $P(H)$, consisting of those $x\in H$ such that $\Delta(x)=x\otimes1+1\otimes x$.  For example, if $\mathfrak{g}$ is a Lie algebra, then the primitives of its universal enveloping algebra are a copy of $\mathfrak{g}$ itself.  Under certain conditions on $H$ which are satisfied in the cases $H=\sym^1,\sym^{-1},\sym'$, the space of primitives can be computed as the perpendicular space to $I^2$, where $I$ is the algebra ideal of positively graded elements \cite{ZelFinite}.  In $\sym^1$, the primitives are the classical power sum functions; in the characteristic zero case, viewed as characters of symmetric groups, their products represent scalar multiples of the indicator functions on corresponding conjugacy classes of $S_n$.  The space of primitives of $\sym^1$ is one-dimensional in each positive degree.  In Subsection \ref{subsec-primitives}, we compute and study the space of primitives of $\sym^{-1}$.  This space is one-dimensional in even degrees and in degree 1, and zero in other degrees.  Its even degree parts generate the center (which coincides with the supercenter) of $\sym^{-1}$ as an algebra.

The most remarkable basis of $\sym^1$ is the basis of Schur functions.  In terms of power series they are described as generating functions for semistandard Young tableaux of a given shape or as a ratio of determinants; in terms of the nilHecke ring they are the result of applying the longest divided difference operator to a single monomial; in terms of the symmetric group they are the images of the characters of irreducible representations; and in terms of $GL(n)$, they are the characters of the irreducible polynomial representations.  As a result, they are an orthonormal integral basis of $\sym^1$ \cite{Mac}.  Subsection \ref{subsec-schur} constructs odd analogues of the Schur functions from the combinatorial perspective.  These functions are defined by using odd Kostka numbers which count semistandard Young tableaux with signs, a slight generalization of a notion studied by Stanley \cite{StanleySignImbalance} and others \cite{Kim}, \cite{Lam2}.  Using results of Reifegerste \cite{Reifegerste} and Sj\"{o}strand \cite{Sjostrand} on the behavior of talbeau sign under the RSK correspondence, we are able to prove that the odd Schur functions are all orthogonal and of norm $\pm1$.  In the sequel to this paper \cite{EKL}, we and Aaron Lauda construct the odd Schur functions inside the odd nilHecke ring in terms of odd divided difference operators.

In Section \ref{sec-rsk}, we give a new proof of the signed RSK correspondence.  Finally, Section \ref{sec-data} is an appendix consisting of some numerical data related to $\sym'$ and $\sym^{-1}$.

The definition of $\sym^{-1}$ is just one step beyond the work \cite{ThibonUng}
of Thibon and Ung. It seems that, despite the multitude of papers on noncommutative 
and quasi-symmetric functions, very little has been written about the quantum 
case, and we did not find any mention of $\sym^{-1}$ and its interesting 
structure in the literature.

%
\subsection{Acknowledgments}
%

A.P.E. would like to thank the NSF for support under the 
Graduate Research Fellowship Program.
M.K. was a visitor to the Simons Center for Geometry and Physics
during the spring semester of 2011. He would like to thank the Center 
and its director John Morgan for their generous support and hospitality.  M.K.
received partial support from the NSF grant DMS-1005750 while working on this paper.   Both authors thank Darij Grinberg and Aaron Lauve for corrections to a previous version of this paper, as well as Thomas Lam for pointing out previous work on Young tableau signs \cite{Kim}, \cite{Lam2}, \cite{Reifegerste}, \cite{Sjostrand}, \cite{StanleySignImbalance} and the connection with the spin-weight domino symmetric functions of Carr\'{e} and Leclerc \cite{CarreLeclerc}, \cite{Lam}.

%
\section{The definition of the $q$-Hopf algebra $\sym$}
%

%
\subsection{The category of $q$-vector spaces and a bilinear form}\label{subsec-q-vector-spaces}
%

We work over a commutative ground ring $\Bbbk$.  
Let $\sym'$ be a free associative $\Z$-graded 
$\Bbbk$-algebra with generators $h_1, h_2, \dots$
(it is convenient to assume $h_0=1$ and $h_i=0$ for $i<0$) 
of degrees $\deg(h_n) = n$. The grading 
gives a vector space decomposition 
$$ \sym' = \oplusop{n\ge 0} \sym'_n.$$ 
We choose $q\in \Bbbk$ and define a multiplication in $\sym'^{\otimes 2}$ by 
$$ (x_1 \otimes x_2) (y_1 \otimes y_2) = q^{\deg(x_2)\deg(y_1)} x_1 y_1 \otimes 
x_2 y_2$$ 
on homogeneous elements. If $q$ is invertible, consider the braided monoidal 
category $\Bbbk\mathrm{-gmod}_q$ 
of $\Z$-graded $\Bbbk$-modules with the braiding functor given by 
$$ V\otimes W \lra W \otimes V, \quad v\otimes w \longmapsto q^{\deg(v)\deg(w)} 
 w\otimes v $$ 
on homogeneous elements. The braiding structure is symmetric if $q\in \{ 1, -1\}$. 
For $q$ not invertible, the above gives $\Bbbk\text{-gmod}_q$ a lax braided structure.   Following \cite{AguiarMahajan}, we refer to bialgebra and Hopf algebra objects in $\Bbbk\text{-gmod}_q$ as $q$-bialgebras and $q$-Hopf algebras, respectively.

The algebra $\sym'$ can be made into a $q$-bialgebra by defining the comultiplication on generators to be
$$ \Delta(h_n) = \sum_{m=0}^n h_m \otimes h_{n-m}. $$  
Given our convention on the $h$'s, we can also write 
$$ \Delta(h_n) = \sum_{m\in \Z} h_m \otimes h_{n-m}. $$ 
The braiding structure implies, for instance, that 
$$ \Delta(h_n h_k) = \sum_{m,r} (h_m \otimes h_{n-m})(h_r \otimes h_{k-r})  
 =  \sum_{m,r} q^{(n-m)r} h_m h_r \otimes h_{n-m}h_{k-r}.  $$
The counit is the obvious one, with $\epsilon(x) =0$ if $\deg(x)>0$.   The $q$-bialgebra $\sym'$ is cocommutative if and only if $q=1$.  

\vspace{0.07in}

For a sequence (or composition) $\alpha=(a_1, \dots, a_k)$, define 
$h_\alpha = h_{a_1}\dots h_{a_k}$. Let also $|\alpha| = a_1+\dots + a_k$ 
and $ S_{\alpha} = S_{a_1}\times \dots \times S_{a_k} \subset S_{|\alpha|}$ 
be the parabolic subgroup of the symmetric group $S_{|\alpha|}$ associated to the composition $\alpha$. 

For compositions $\alpha,\beta$ such that $|\alpha|=|\beta|=n$, denote by 
$${}_{\beta}S_{\alpha} \ = \  S_{\beta} \backslash  S_n / S_{\alpha} $$ 
the set of double cosets of the subgroups $S_{\alpha}$ and $S_{\beta}$ in $S_n$. 
Each double coset $c$ has a unique minimal length representative $\sigma(c)\in S_n$. 
Denote by $\ell(c) = \ell(\sigma(c))$ the length of this representative. 

Elements of $S_n$ admit a graphical description. A permutation 
$\sigma\in S_n$  can be presented by $n$ curves in the plane 
connecting points $1, \dots, n$ on a horizontal line to points 
$\sigma(1),\dots, \sigma(n)$ on a parallel line above the former line such that
\begin{itemize}
\item The curves have no critical points with respect to the height function (that is, they never flatten out or turn around),
\item Any two curves intersect at most once (the curves starting at $i<j$ intersect if and only if $\sigma(i)>\sigma(j)$),
\item There are no triple intersections.
\end{itemize}
Such diagrams are considered as their combinatorial type, that is, up to rel boundary 
homotopy 
through diagrams satisfying the conditions above.  If the curves are in general position, then the length 
$\ell(\sigma)$ is the number of intersection points of these curves, equal to the 
number of pairs $i<j$ such that $\sigma(i)>\sigma(j)$.
 
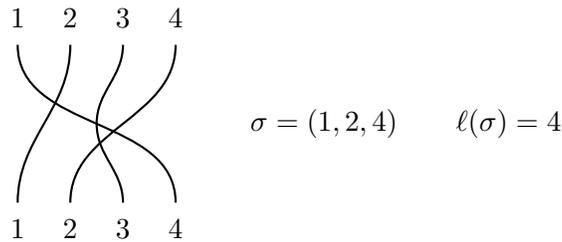
\begin{equation*}
\hackcenter{\begin{tikzpicture}[scale=.7]
    \draw[thick] (1,0) [out=90, in=-90] to (2,3);
    \draw[thick] (2,0) [out=90, in=-90] to (4,3);
    \draw[thick] (3,0) [out=90, in=-90] to (2.5,1.5) [out=90, in=-90] to (3,3);
    \draw[thick] (4,0) [out=90, in=-90] to (1,3);
    \node at (1,-.5) {$1$};
    \node at (2,-.5) {$2$};
    \node at (3,-.5) {$3$};
    \node at (4,-.5) {$4$};
    \node at (1,3.5) {$1$};
    \node at (2,3.5) {$2$};
    \node at (3,3.5) {$3$};
    \node at (4,3.5) {$4$};
\end{tikzpicture}}
\qquad\sigma=(1,2,4)\qquad\ell(\sigma)=4
\end{equation*}

Minimal double coset representatives $\sigma(c),$ $c\in {}_{\beta} S_{\alpha}$ 
are singled out by the following condition. Draw intervals (or platforms) of ``size'' 
$a_1, \dots, a_k$ 
from left to right at the bottom of the permutation diagram and platforms 
$b_1, \dots, b_r$ at the top ($\alpha=(a_1, \dots, a_k)$, $\beta=(b_1,\dots,b_r)$), 
so that the 
first $a_1$ lines from the left start off at the first bottom platform, the next 
$a_2$ lines at the second platform, and so forth, and likewise for the top platforms. Then 
$\sigma= \sigma(c)$ for some double coset $c$ if and only if any two lines that 
start or end in the same platform do not intersect. An example is depicted below, 
with $\alpha=(4,1,2,2),$ $\beta=(2,5,2)$, $\ell(\sigma) =9$. 

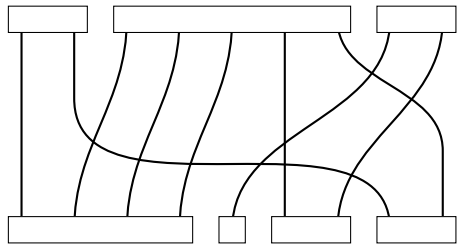
\begin{equation*}
\hackcenter{\begin{tikzpicture}[scale=.7]
    \draw[thick] (1,0) [out=90, in=-90] to (1,4);
    \draw[thick] (2,0) [out=90, in=-90] to (3,4);
    \draw[thick] (3,0) [out=90, in=-90] to (4,4);
    \draw[thick] (4,0) [out=90, in=-90] to (5,4);
    \draw[thick] (5,0) [out=90, in=-90] to (8,4);
    \draw[thick] (6,0) [out=90, in=-90] to (6,4);
    \draw[thick] (7,0) [out=90, in=-90] to (9,4);
    \draw[thick] (8,0) [out=90, in=-90] to (2, 2.5) [out=90, in=-90] to (2,4);
    \draw[thick] (9,0) [out=90, in=-90] to (9, 1.5) [out=90, in=-90] to (7,4);
    \filldraw[draw=black, fill=white] (.75,-.25) rectangle (4.25,.25);
    \filldraw[draw=black, fill=white] (4.75,-.25) rectangle (5.25,.25);
    \filldraw[draw=black, fill=white] (5.75,-.25) rectangle (7.25,.25);
    \filldraw[draw=black, fill=white] (7.75,-.25) rectangle (9.25,.25);
    \filldraw[draw=black, fill=white] (.75,3.75) rectangle (2.25,4.25);
    \filldraw[draw=black, fill=white] (2.75,3.75) rectangle (7.25,4.25);
    \filldraw[draw=black, fill=white] (7.75,3.75) rectangle (9.25,4.25);
\end{tikzpicture}}
\end{equation*}

If $\alpha=(1,2,1)$ and $\beta= (2,2)$, there are four cosets, with minimal length
representatives $1, (2,3,4), (1,3,2), $ and $(1,3,4,2)$: 

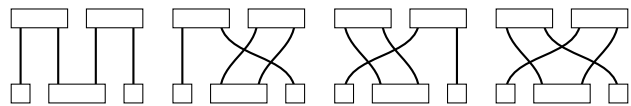
\begin{equation*}
\hackcenter{\begin{tikzpicture}[scale=.5]
    \draw[thick] (1,0) [out=90, in=-90] to (1,2);
    \draw[thick] (2,0) [out=90, in=-90] to (2,2);
    \draw[thick] (3,0) [out=90, in=-90] to (3,2);
    \draw[thick] (4,0) [out=90, in=-90] to (4,2);
    \filldraw[draw=black, fill=white] (.75,-.25) rectangle (1.25,.25);
    \filldraw[draw=black, fill=white] (1.75,-.25) rectangle (3.25,.25);
    \filldraw[draw=black, fill=white] (3.75,-.25) rectangle (4.25,.25);
    \filldraw[draw=black, fill=white] (.75,1.75) rectangle (2.25,2.25);
    \filldraw[draw=black, fill=white] (2.75,1.75) rectangle (4.25,2.25);
\end{tikzpicture}\quad
\begin{tikzpicture}[scale=.5]
    \draw[thick] (1,0) [out=90, in=-90] to (1,2);
    \draw[thick] (2,0) [out=90, in=-90] to (3,2);
    \draw[thick] (3,0) [out=90, in=-90] to (4,2);
    \draw[thick] (4,0) [out=90, in=-90] to (2,2);
    \filldraw[draw=black, fill=white] (.75,-.25) rectangle (1.25,.25);
    \filldraw[draw=black, fill=white] (1.75,-.25) rectangle (3.25,.25);
    \filldraw[draw=black, fill=white] (3.75,-.25) rectangle (4.25,.25);
    \filldraw[draw=black, fill=white] (.75,1.75) rectangle (2.25,2.25);
    \filldraw[draw=black, fill=white] (2.75,1.75) rectangle (4.25,2.25);
\end{tikzpicture}\quad
\begin{tikzpicture}[scale=.5]
    \draw[thick] (1,0) [out=90, in=-90] to (3,2);
    \draw[thick] (2,0) [out=90, in=-90] to (1,2);
    \draw[thick] (3,0) [out=90, in=-90] to (2,2);
    \draw[thick] (4,0) [out=90, in=-90] to (4,2);
    \filldraw[draw=black, fill=white] (.75,-.25) rectangle (1.25,.25);
    \filldraw[draw=black, fill=white] (1.75,-.25) rectangle (3.25,.25);
    \filldraw[draw=black, fill=white] (3.75,-.25) rectangle (4.25,.25);
    \filldraw[draw=black, fill=white] (.75,1.75) rectangle (2.25,2.25);
    \filldraw[draw=black, fill=white] (2.75,1.75) rectangle (4.25,2.25);
\end{tikzpicture}\quad
\begin{tikzpicture}[scale=.5]
    \draw[thick] (1,0) [out=90, in=-90] to (3,2);
    \draw[thick] (2,0) [out=90, in=-90] to (1,2);
    \draw[thick] (3,0) [out=90, in=-90] to (4,2);
    \draw[thick] (4,0) [out=90, in=-90] to (2,2);
    \filldraw[draw=black, fill=white] (.75,-.25) rectangle (1.25,.25);
    \filldraw[draw=black, fill=white] (1.75,-.25) rectangle (3.25,.25);
    \filldraw[draw=black, fill=white] (3.75,-.25) rectangle (4.25,.25);
    \filldraw[draw=black, fill=white] (.75,1.75) rectangle (2.25,2.25);
    \filldraw[draw=black, fill=white] (2.75,1.75) rectangle (4.25,2.25);
\end{tikzpicture}\quad}
\end{equation*}

Fix $q\in \Bbbk$ and 
define a symmetric $\Bbbk$-bilinear form on $\sym'$ taking values in $\Bbbk$ by 
\begin{equation}
\label{eqn-bilinform} 
(h_{\beta}, h_{\alpha}) = \begin{cases}\sum_{c\in {}_{\beta} S_{\alpha}}
q^{\ell(c)}&\text{if }|\alpha|=|\beta|,\\ 
0&\text{otherwise}.\end{cases}
\end{equation}
  
The weight spaces $\sym'_n$ of degree $n$ are pairwise orthogonal relative to this form. 

\vspace{0.07in} 

\begin{example} The inner product $(h_2h_2, h_1h_2h_1) = 1 + 2 q^2 + q^3$; see 
the four diagrams above. Each double coset in ${}_{\beta} S_{\alpha}$ contributes 
$q$ to the power equal to the number of crossings in the diagram of the coset. 
\end{example}

\vspace{0.07in} 

We extend this form to $\sym'^{\otimes 2}$ by 
\begin{equation}\label{whereisq} 
(y_1\otimes y_2, x_1\otimes x_2) = (y_1,x_1) (y_2, x_2) .
\end{equation}  
One may wonder why the factor $q^{\deg(y_2)\deg(x_1)}$ does not appear in 
this formula given that $y_2$ seems to move past $x_1$.  
Powers of $q$ are also absent 
in \cite[Proposition 1.2.3]{Lus4} in a very similar situation.  
The graphical interpretation of the bilinear form provides a reason: 
we think of the tensor product of elements as occurring horizontally, 
while the diagrams 
used in the computation of the bilinear form occur vertically. The picture below 
shows a diagram contributing to the inner product $(y_1\otimes y_2, x_1\otimes x_2)$ 
for suitable $x_1,x_2,y_1,y_2$. This diagram is a disjoint union of two diagrams, 
the one on the left contributing to $(y_1,x_1)$, the right one contributing to 
$(y_2,x_2)$. 

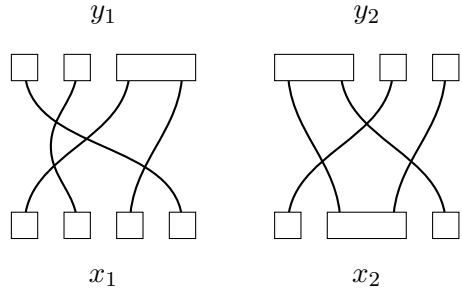
\begin{equation*}
\hackcenter{\begin{tikzpicture}[scale=0.7]
    \draw[thick] (1,0) [out=90, in=-90] to (3,3);
    \draw[thick] (2,0) [out=90, in=-90] to (1.5,1.5) [out=90, in=-90] to (2,3);
    \draw[thick] (3,0) [out=90, in=-90] to (4,3);
    \draw[thick] (4,0) [out=90, in=-90] to (1,3);
    \draw[thick] (6,0) [out=90, in=-90] to (8,3);
    \draw[thick] (7,0) [out=90, in=-90] to (6,3);
    \draw[thick] (8,0) [out=90, in=-90] to (9,3);
    \draw[thick] (9,0) [out=90, in=-90] to (7,3);
    \filldraw[draw=black, fill=white] (.75,-.25) rectangle (1.25,.25);
    \filldraw[draw=black, fill=white] (1.75,-.25) rectangle (2.25,.25);
    \filldraw[draw=black, fill=white] (2.75,-.25) rectangle (3.25,.25);
    \filldraw[draw=black, fill=white] (3.75,-.25) rectangle (4.25,.25);
    \filldraw[draw=black, fill=white] (5.75,-.25) rectangle (6.25,.25);
    \filldraw[draw=black, fill=white] (6.75,-.25) rectangle (8.25,.25);
    \filldraw[draw=black, fill=white] (8.75,-.25) rectangle (9.25,.25);
    \filldraw[draw=black, fill=white] (.75,2.75) rectangle (1.25,3.25);
    \filldraw[draw=black, fill=white] (1.75,2.75) rectangle (2.25,3.25);
    \filldraw[draw=black, fill=white] (2.75,2.75) rectangle (4.25,3.25);
    \filldraw[draw=black, fill=white] (5.75,2.75) rectangle (7.25,3.25);
    \filldraw[draw=black, fill=white] (7.75,2.75) rectangle (8.25,3.25);
    \filldraw[draw=black, fill=white] (8.75,2.75) rectangle (9.25,3.25);
    \node at (2.5,-1) {$x_1$};
    \node at (2.5,4) {$y_1$};
    \node at (7.5,-1) {$x_2$};
    \node at (7.5,4) {$y_2$};
\end{tikzpicture}}
\end{equation*}

No strands from distinct tensor factors ever cross, justifying 
equation (\ref{whereisq}). From this viewpoint it 
would be more natural to write $\binom{y}{x}$ rather than $(y,x)$; we will not do so 
for obvious reasons. In this notation, equation (\ref{whereisq}) would become 
$$ \binom{y_1\otimes y_2}{x_1\otimes x_2} \ = \ 
 \binom{y_1}{x_1} \ \binom{y_2}{x_2} ,$$ 
with no change in the relative position of the four variables on the two sides 
of the equation. 

\begin{prop} For all $x,y_1,y_2\in\sym'$, 
\begin{equation}
(y_1\otimes y_2,\Delta(x)) = (y_1y_2,x). 
\end{equation}
\end{prop}  
\noindent In other words, multiplication and comultiplication are adjoint operators 
relative to these forms on $\sym'$ and $\sym'^{\otimes 2}$.  

\begin{proof} It is enough to check the adjointness when $x,y_1,y_2$ are products 
of $h_n$'s. The inner product $(y_1y_2,x)$ is computed as a sum over diagrams 
(of double cosets) with platforms at the bottom corresponding to the terms of $x$ 
and platforms at the top corresponding to those of $y_1$ followed by those of $y_2$. 
In a given diagram, lines from each platform of $x$ will split into those 
going into $y_1,$ respectively $y_2$, platforms. These two types of lines will 
intersect, and the intersection points will contribute powers of $q$, which 
are matched by the powers of $q$ in $\Delta(x)$ coming from the definition 
of multiplication in $\sym'^{\otimes 2}$; see the diagram below.  \end{proof}

\begin{equation*}\begin{split}
\qquad\text{intersection points between two types of lines}:\\
\hackcenter
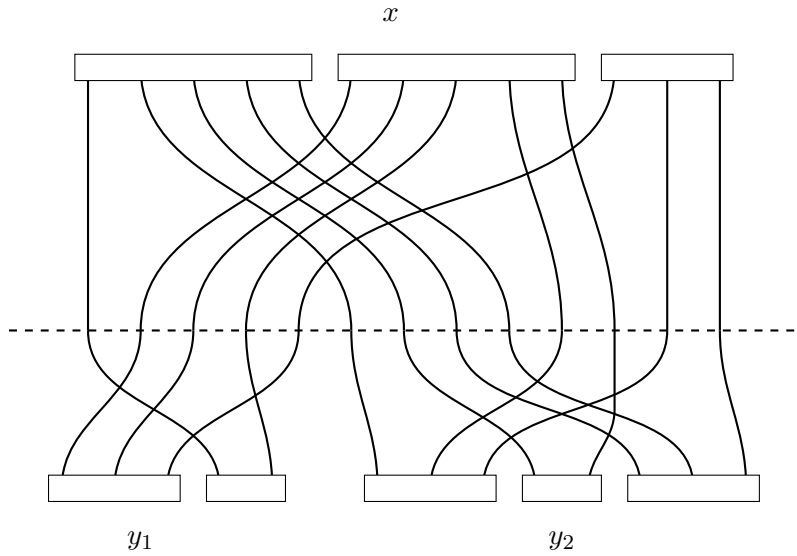
{\begin{tikzpicture}[scale=.7]
    \draw[thick] (1,0) [out=90, in=-90] to (2.5,3);
    \draw[thick] (2,0) [out=90, in=-90] to (3.5,3);
    \draw[thick] (3,0) [out=90, in=-90] to (5.5,3);
    \draw[thick] (4,0) [out=90, in=-90] to (1.5,3);
    \draw[thick] (5,0) [out=90, in=-90] to (4.5,3);
    \draw[thick] (7,0) [out=90, in=-90] to (6.5,3);
    \draw[thick] (8,0) [out=90, in=-90] to (10.5,3);
    \draw[thick] (9,0) [out=90, in=-90] to (12.5,3);
    \draw[thick] (10,0) [out=90, in=-90] to (7.5,3);
    \draw[thick] (11,0) [out=90, in=-90] to (11.5,1.5) [out=90, in=-90] to (11.5,3);
    \draw[thick] (12,0) [out=90, in=-90] to (8.5,3);
    \draw[thick] (13,0) [out=90, in=-90] to (9.5,3);
    \draw[thick] (14,0) [out=90, in=-90] to (13.5,3);
    \draw[thick, dashed] (0,3) -- (15,3);
    \draw[thick] (1.5,3) [out=90, in=-90] to (1.5,8);
    \draw[thick] (2.5,3) [out=90, in=-90] to (6.5,8);
    \draw[thick] (3.5,3) [out=90, in=-90] to (7.5,8);
    \draw[thick] (4.5,3) [out=90, in=-90] to (8.5,8);
    \draw[thick] (5.5,3) [out=90, in=-90] to (11.5,8);
    \draw[thick] (6.5,3) [out=90, in=-90] to (2.5,8);
    \draw[thick] (7.5,3) [out=90, in=-90] to (3.5,8);
    \draw[thick] (8.5,3) [out=90, in=-90] to (4.5,8);
    \draw[thick] (9.5,3) [out=90, in=-90] to (5.5,8);
    \draw[thick] (10.5,3) [out=90, in=-90] to (9.5,8);
    \draw[thick] (11.5,3) [out=90, in=-90] to (10.5,8);
    \draw[thick] (12.5,3) [out=90, in=-90] to (12.5,8);
    \draw[thick] (13.5,3) [out=90, in=-90] to (13.5,8);
    \filldraw[draw=black, fill=white] (.75,-.25) rectangle (3.25,.25);
    \filldraw[draw=black, fill=white] (3.75,-.25) rectangle (5.25,.25);
    \filldraw[draw=black, fill=white] (6.75,-.25) rectangle (9.25,.25);
    \filldraw[draw=black, fill=white] (9.75,-.25) rectangle (11.25,.25);
    \filldraw[draw=black, fill=white] (11.75,-.25) rectangle (14.25,.25);
    \filldraw[draw=black, fill=white] (1.25,7.75) rectangle (5.75,8.25);
    \filldraw[draw=black, fill=white] (6.25,7.75) rectangle (10.75,8.25);
    \filldraw[draw=black, fill=white] (11.25,7.75) rectangle (13.75,8.25);
    \node at (2.5,-1) {$y_1$};
    \node at (10.5,-1) {$y_2$};
    \node at (7.25,9) {$x$};
\end{tikzpicture}}
\end{split}\end{equation*}

Let $\mc{I}\subset \sym'$ be the radical of $(\cdot,\cdot)$.  Then  
$\mc{I}= \oplusop{n\ge 0}
\mc{I}_n$, where $\mc{I}_n = \mc{I}\cap \sym'_n$. Define 
$\sym= \sym'/\mc{I}$ and let $\sym_n$ denote the subspace of elements of 
$\sym$ which are homogeneous of degree $n$. 

To emphasize the dependence on $q$ one can also write $\sym^q$ instead 
of $\sym$ and $\sym^q_n$ instead of $\sym_n$. We will use the 
shorter notation whenever possible.

\begin{prop} $\mc{I}$ is a $q$-bialgebra ideal in $\sym'$: 
$$ \mc{I}\sym' = \sym'\mc{I} = \mc{I}, \quad \Delta(\mc{I}) \subset 
\mc{I}\otimes  \sym' + \sym' \otimes \mc{I}.$$ 
\end{prop} 

\begin{proof} These properties of $\mc{I}$ follow at once
from adjointness of multiplication and comultiplication. 
\end{proof}

\begin{cor} $\sym$ inherits a $q$-bialgebra structure from that of $\sym'$.
\end{cor} 

If $q=0$, the bilinear form degenerates and $\dim(\sym_n)=1$ for all $n\geq0$.

\vspace{0.07in}

If $q=1$, the inner product $(h_{\beta}, h_{\alpha}) = |{}_{\beta} S_{\alpha}|  $ 
is the number of double cosets, and it coincides with the standard inner product 
on the bialgebra of symmetric functions $\Bbbk[h_1, h_2, \dots ]$ in infinitely many 
variables $x_1, x_2, \dots $, with $h_n$ being the $n$-th complete symmetric 
function. In this case the ideal $\mc{I}$ is generated by commutators 
$[h_n,h_m]= h_n h_m - h_m h_n$ over all $n,m$, and the bialgebra $\sym^1$ is 
the maximal commutative quotient of $\sym'$. Note that we defined $\sym'$ 
as a free associative (not commutative) algebra. The bilinear form in the $q=1$ case forces 
commutativity but nothing else. Nondegeneracy of the form on the maximal commutative
quotient follows from the result that the elements $h_{\lambda}=h_{\lambda_1} 
\dots h_{\lambda_r}$ are linearly independent over all partitions $\lambda$ of $n$. 
This is proved by introducing elementary symmetric functions $e_n$ via the inductive 
relation 
 $$ \sum_{k=0}^n (-1)^k h_k e_{n-k} = 0,$$ 
defining $e_{\lambda} = e_{\lambda_1} \dots e_{\lambda_r}$, 
and then checking that the matrix of the bilinear form is upper-triangular 
with ones on the diagonal with respect to the bases $\{h_{\lambda}\}_{\lambda\vdash n}$ 
and $\{e_{\lambda^T}\}_{\lambda\vdash n}$ for any total order on 
partitions refining the dominance order, where 
$\lambda^T$ is the dual (or transpose) partition of $\lambda$.

%
\subsection{Odd complete and elementary symmetric functions}\label{subsec-e-h}
%

From now on, unless stated otherwise, we take $q=-1$.  In this case we call 
$\sym=\sym^{-1}$ the bialgebra of \textit{odd symmetric functions}.  Choosing 
$q=-1$ makes $\Bbbk\text{-gmod}_q$ the category of $\Z$-graded 
super-vector spaces, so that an algebra in $\Bbbk\text{-gmod}_{-1}$ is a 
$\Z$-graded superalgebra, and likewise for $(-1)$-bialgebras and 
$(-1)$-Hopf algebras.  The super-grading is the mod 2 reduction of the $\Z$-grading.

\vspace{0.07in}

When $q=-1$, equation \eqref{eqn-bilinform} takes the form
\begin{equation}\label{eqn-odd-bilform}
(h_{\beta}, h_{\alpha}) = \begin{cases}\sum_{c\in {}_{\beta} S_{\alpha}}
(-1)^{\ell(c)}&\text{if }|\alpha|=|\beta|,\\ 
0&\text{otherwise}.\end{cases}
\end{equation}
\noindent If $\lambda=(\lambda_1,\ldots,\lambda_r)$ is a partition, the product
\begin{equation*}
h_\lambda=h_{\lambda_1}h_{\lambda_2}\cdots h_{\lambda_r}
\end{equation*}
is called an \textit{odd complete symmetric function}.  By analogy with the even ($q=1$) case, inductively define elements $e_n\in\sym$ by $e_0=1$ and
\begin{equation}\label{eqn-defn-e}
\sum_{k=0}^n(-1)^{\binom{k+1}{2}}e_kh_{n-k}=0.
\end{equation}
Equation \eqref{eqn-defn-e} is equivalent to the equation
\begin{equation*}
\sum_{k=0}^n(-1)^{\binom{k+1}{2}}h_{n-k}e_k=0.
\end{equation*}
This can be checked by a straightforward calculation or by applying the 
involution $\psi_1\psi_2$ of Subsection \ref{subsec-automorphisms}.  
The \textit{odd elementary symmetric functions} are defined to be products
\begin{equation*}
e_\lambda=e_{\lambda_1}e_{\lambda_2}\cdots e_{\lambda_r}
\end{equation*}
for partitions $\lambda=(\lambda_1,\ldots,\lambda_r)$.

Define $h_\alpha$ and $e_\alpha$ for a composition $\alpha=(a_1,\ldots,a_r)$ 
similarly, as
\begin{equation*}
h_\alpha=h_{a_1}\cdots h_{a_r}, \ \ \ \ 
e_\alpha=e_{a_1}\cdots e_{a_r}.
\end{equation*}
We will write 
$\binom{alpha}{2}=\binom{a_1}{2}+\ldots+\binom{a_r}{2}$ and call $\ell(\alpha)=r$ 
the length of the composition $\alpha$.  
An easy inductive argument shows that another equivalent definition of $e_n$ is 
\begin{equation}\label{eqn-defn-e-alt}
e_n=(-1)^{\binom{n+1}{2}}\sum_{|\alpha|=n}(-1)^{\ell(\alpha)}h_\alpha.
\end{equation}
The sum is over all $2^{n-1}$ compositions $\alpha$ of $n$.
Observe that
\begin{equation*}
(-1)^{\binom{\lambda}{2}+|\lambda|}(-1)^{\lambda_2^T+\lambda_4^T+\lambda_6^T+\ldots}=(-1)^{|\lambda|}
\end{equation*}
for any partition $\lambda$.  Here, $\lambda_i^T$ means the $i$-th row length of $\lambda^T$ (equivalently, the $i$-th column height of $\lambda$).  This will be useful later in studying odd Schur functions.

Equation \eqref{eqn-defn-e} can be solved for $e_n$ in terms of $h_1,h_2,\ldots,h_n$, so \eqref{eqn-defn-e} makes sense as a definition of $e_n$.  The first few $e_n$ are
\begin{equation*}\begin{split}
e_1&=h_1, \\
e_2&=h_2-h_1^2, \\
e_3&=h_3-h_1^3, \\
e_4&=-h_4+h_2^2-h_2h_1^2+h_1^4, \\
e_5&=h_5-2h_4h_1-h_3h_1^2+h_2^2h_1+h_1^5.
\end{split}\end{equation*}
Since equation \eqref{eqn-defn-e} also allows one to solve for $h_n$ in terms of $e_1,e_2,\ldots,e_n$, any element of $\sym$ is a linear combination of words in the $e_k$. It is 
convenient to set $e_k=0$ for $k<0$. 
\begin{prop}\label{prop-h} We have the following:
\begin{enumerate}
\item The comultiplication on $e_n$ is 
\begin{equation}\label{eqn-h-coproduct}
\Delta(e_n) = \sum_{k=0}^ne_k\otimes e_{n-k}.
\end{equation}
\item If $\alpha$ is a composition of $n$, then
\begin{equation}\label{eqn-h-e-bil-form}
(h_{\alpha},e_n)=
\begin{cases}1&\text{if }\alpha=(1,1,\ldots,1),\\0&\text{otherwise.}\end{cases}
\end{equation}\end{enumerate}\end{prop}
\noindent Since $(\cdot,\cdot)$ is nondegenerate and $\sym$ is finite dimensional in each degree, property \eqref{eqn-h-e-bil-form} uniquely characterizes the elements $e_n\in\sym$.

\begin{proof} We prove both statements by simultaneous induction on $n$, the 
cases $n=0,1$ being clear.  To prove the second statement, it suffices to prove
\begin{equation*}
(h_mx,e_n)=\begin{cases}(x,e_{n-1})&\text{if }m=1,\\0&\text{otherwise}
\end{cases}
\end{equation*}
for any $x\in\sym$.  First, a calculation: for $k<n$, the inductive hypothesis implies
\begin{equation}\label{eqn-prop-h-calc}
(h_mx,e_kh_{n-k})=(-1)^{km}(x,e_kh_{n-k-m})+(-1)^{(k-1)(m-1)}
(x,e_{k-1}h_{n-k-m+1}).
\end{equation}
To derive this equation, it is useful to first extend the bilinear form diagrammatics of Subsection \ref{subsec-q-vector-spaces}: we  represent $h_n$'s by white platforms of size $n$ 
and $e_n$'s by black platforms of size $n$.  Now we start proving \eqref{eqn-prop-h-calc} by drawing $e_k h_{n-k}$ below $h_m x$. 

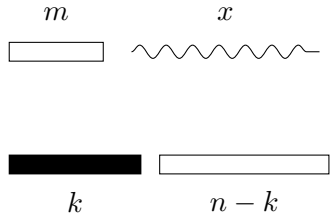
\begin{equation*}
\hackcenter{\begin{tikzpicture}[scale=.5,decoration=snake]
    \filldraw[draw=black, fill=black] (.75,-.25) rectangle (4.25,.25);
    \filldraw[draw=black, fill=white] (4.75,-.25) rectangle (9.25,.25);
    \filldraw[draw=black, fill=white] (.75,2.75) rectangle (3.25,3.25);
    \draw[decorate] (4,3) -- (9,3);
    \node at (2.5,-1) {$k$};
    \node at (7,-1) {$n-k$};
    \node at (2,4) {$m$};
    \node at (6.5,4) {$x$};
\end{tikzpicture}}
\end{equation*}

Strictly speaking, a diagram containing a black platform representing $e_k$ stands in for a linear combination of diagrams in which the black platform is replaced by groups of white platforms which come from writing $e_k$ as a linear combination of $h_\alpha$'s.

By the inductive hypothesis applied to $k<n$, 
at most one line can connect the bottom left black platform of width 
$k$ (representing $e_k$) with the top left white platform of width 
$m$ (representing $h_m$). If no lines connect these two platforms, 
all lines from $h_m$ will be connected to $h_{n-k}$ (necessarily requiring 
$n-k\ge m$), while all lines from $e_k$ will go into $x$, creating 
$km$ intersection points that contribute $(-1)^{km}$; see below. 
The contribution from these diagrams will total
$(-1)^{km}(x,e_kh_{n-k-m})$. The dotted curve in the figure below encloses the 
area producing the factor $(x,e_k h_{n-k-m})$. 

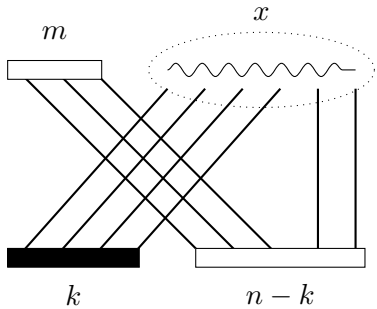
\begin{equation*}
\hackcenter{\begin{tikzpicture}[scale=.5,decoration=snake]
    \draw[thick] (1,0) -- (5,4.5);
    \draw[thick] (2,0) -- (6,4.5);
    \draw[thick] (3,0) -- (7,4.5);
    \draw[thick] (4,0) -- (8,4.5);
    \draw[thick] (6,0) -- (1,5);
    \draw[thick] (7,0) -- (2,5);
    \draw[thick] (8,0) -- (3,5);
    \draw[thick] (9,0) -- (9,4.5);
    \draw[thick] (10,0) -- (10,4.5);
    \filldraw[draw=black, fill=black] (.75,-.25) rectangle (4.25,.25);
    \filldraw[draw=black, fill=white] (5.75,-.25) rectangle (10.25,.25);
    \filldraw[draw=black, fill=white] (.75,4.75) rectangle (3.25,5.25);
    \draw[decorate] (5,5) -- (10,5);
    \draw[dotted] (7.5,5) ellipse (3 and 1);
    \node at (2.5,-1) {$k$};
    \node at (8,-1) {$n-k$};
    \node at (2,6) {$m$};
    \node at (7.5,6.5) {$x$};
\end{tikzpicture}}
\end{equation*}

If one line connects the $e_k$ with the $h_m$ platform, the remaining 
$k-1$ lines from the black platform go into $x$, while $m-1$ lines 
from $h_m$  enter $h_{n-k}$. These two types of lines intersect 
and contribute $(-1)^{(k-1)(m-1)}$ to the sum. In the diagram below 
we denote each of these bunches of ``parallel'' lines by a single line 
labelled $k-1$, respectively $m-1$. The dotted curve below encloses the area 
contributing the factor $(x,e_{k-1}h_{n-m-k+1})$. 

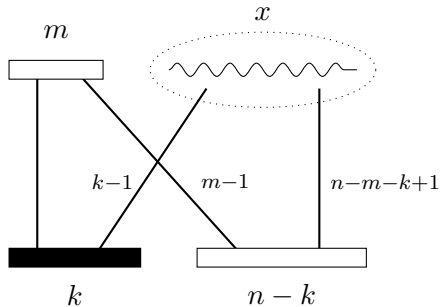
\begin{equation*}
\hackcenter{\begin{tikzpicture}[scale=.5,decoration=snake]
    \draw[thick] (1.5,0) -- (1.5,5);
    \draw[thick] (3,0) -- (6,4.5);
    \draw[thick] (7,0) -- (2.5,5);
    \draw[thick] (9,0) -- (9,4.5);
    \filldraw[draw=black, fill=black] (.75,-.25) rectangle (4.25,.25);
    \filldraw[draw=black, fill=white] (5.75,-.25) rectangle (10.25,.25);
    \filldraw[draw=black, fill=white] (.75,4.75) rectangle (3.25,5.25);
    \draw[decorate] (5,5) -- (10,5);
    \draw[dotted] (7.5,5) ellipse (3 and 1);
    \node at (2.5,-1) {$k$};
    \node at (8,-1) {$n-k$};
    \node at (2,6) {$m$};
    \node at (7.5,6.5) {$x$};
    \node at (3.5,2) {$\scriptstyle k-1$};
    \node at (6.5,2) {$\scriptstyle m-1$};
    \node at (10.75,2) {$\scriptstyle n-m-k+1$};
\end{tikzpicture}}
\end{equation*}

This computation proves (\ref{eqn-prop-h-calc}). Therefore
\begin{equation*}\begin{split}
-(-1)^{\binom{n+1}{2}}(h_mx, e_n)&\refequal{\eqref{eqn-defn-e}}
\sum_{k=0}^{n-1}(-1)^{\binom{k+1}{2}}(h_mx,e_kh_{n-k})\\
&\refequal{\eqref{eqn-prop-h-calc}}\sum_{k=0}^{n-1}(-1)^{\binom{k+1}{2}}
\left[(-1)^{km}(x,e_kh_{n-k-m})+(-1)^{(k-1)(m-1)}(x,e_{k-1}h_{n-k-m+1})
\right]\\
&=(-1)^{\binom{n}{2}+m(n-1)}(x,e_{n-1}h_{1-m}).
\end{split}\end{equation*}
The third equality follows because all terms but one cancel in pairs.  Since $h_i=0$ 
for $i<0$, the second statement follows.  The first statement follows from the 
second, since adjointness of multiplication and comultiplication implies (recall 
that the bilinear form is symmetric) 
\begin{equation*}
(\Delta(e_n),h_\lambda\otimes h_\mu)=(e_n,h_\lambda h_\mu)=\begin{cases}1
&\lambda=(1^k),\mu=(1^\ell),k+\ell=n,\\0&\text{otherwise.}\end{cases}
\end{equation*}\end{proof}

\vspace{0.07in}

The proposition implies that unlike the $h_n$'s, the $e_n$'s do not all 
have norm 1.  The elements $e_0=1$ and $e_1=h_1$ both have norm 1, and 
\begin{eqnarray*}
-(-1)^{\binom{n+1}{2}}(e_n,e_n)	=&\sum_{k=0}^{n-1}(-1)^{\binom{k+1}{2}}(e_n,e_kh_{n-k})		&\eqref{eqn-defn-e}		\\
						=&(-1)^{\binom{n}{2}}(e_n,e_{n-1}h_1)						&\eqref{eqn-h-e-bil-form}	\\
						=&(-1)^{\binom{n}{2}}(\Delta(e_n),e_{n-1}\otimes h_1)			&					\\
						=&(-1)^{\binom{n}{2}}(e_{n-1},e_{n-1})						&\eqref{eqn-h-coproduct},\eqref{eqn-h-e-bil-form}.
\end{eqnarray*}
Solving the resulting recurrence relation, we find
\begin{equation}\label{eqn-e-norm}
(e_n,e_n)=(-1)^{\binom{n}{2}}.
\end{equation}

The bilinear form can be evaluated on products of $h$'s and $e$'s with the 
help of diagrammatics as follows (these diagrammatics parallel the graphical calculus developed in 
\cite{Cvitanovic} in the even case).  Let $h_n^+=h_n$ and $h_n^-=e_n$.  
Let $\alpha=(a_1,\ldots, a_r), \beta=(b_1,\ldots,b_s)$ be compositions of $n$ 
and let $\epsilon,\eta$ be tuples of signs of lengths $r,s$ respectively.  
We want to compute
\begin{equation*}
(h_\beta^\eta,h_\alpha^\epsilon)=(h_{b_1}^{\eta_1}\cdots h_{b_s}^{\eta_s},
h_{a_1}^{\epsilon_1}\cdots h_{a_r}^{\epsilon_r}).
\end{equation*}
In a rectangular region, draw platforms of widths $a_1,\ldots,a_r$ along the 
bottom and of widths $b_1,\ldots,b_s$ along the top.  Color a platform white 
(respectively black) if its corresponding sign is $1$ (respectively $-1$).  
Then connect platforms by strands subject to the following rules:
\begin{itemize}
\item As described in the previous subsection, a platform of width $k$ has 
$k$ strands attached to it, and strands are generic curves (no height critical 
points, no triple intersections),
\item The depicted permutation is a minimal double coset representative 
for $S_\beta\backslash S_n/S_\alpha$,
\item For all black platforms $P$ and white platforms $P'$, there is at most 
one strand connecting $P$ to $P'$,
\item Such diagrams are considered up to identification in $S_n$.  That is, diagrams are considered modulo rel boundary homotopy through generic diagrams, and Reidemeister III moves.
\end{itemize}
Let $\text{diag}(\beta,\alpha,\eta,\epsilon)$ be the set of all such diagrams, up to 
the described equivalence.  To each diagram 
$D\in\text{diag}(\beta,\alpha,\eta,\epsilon)$ representing a permutation 
$\sigma\in S_n$, assign a sign $\pm1$ in the following way:
\begin{itemize}
\item Assign a sign $(-1)^{\ell(\sigma)}$, where $\ell(\sigma)$ is the Coxeter 
length of $\sigma$.  That is, $\ell(\sigma)$ equals the number of crossings in the 
minimal double coset representative diagram; equivalently, $\ell(\sigma)$ is the number 
of pairs $i<j$ such that $\sigma(i)>\sigma(j)$.
\item For each pair of black platforms, assign a sign $(-1)^{\binom{k}{2}}=(-1)^{\frac{1}{2}k(k-1)}$, where 
$k$ is the number of strands connecting these two platforms (by equation \ref{eqn-e-norm}).
\end{itemize}
For each diagram $D\in\text{diag}(\beta,\alpha,\eta,\epsilon)$, let $\sign(D)$ be 
this sign (the product of the two factors just described).  The results above imply the 
following.
\begin{prop} The product $(h_\beta^\eta,h_\alpha^\epsilon)$ is given by
\begin{equation*}
(h_\beta^\eta,h_\alpha^\epsilon)=
\sum_{D\in\text{diag}(\beta,\alpha,\eta,\epsilon)}\sign(D).
\end{equation*}\end{prop}

\vspace{0.07in}

\begin{example} Consider the product $(e_2h_1h_2,h_2e_3)=1-1-1=-1$.  The 
three contributing diagrams and their signs are shown below. Each diagram has 
even number of crossings, and the nontrivial signs come
from having two black boxes connected by a pair of lines in the second and 
third diagrams. 

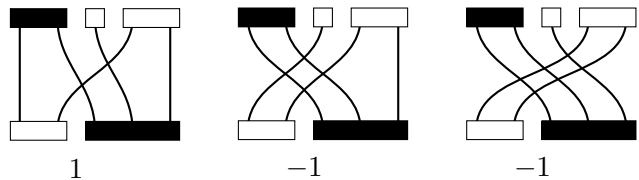
\begin{equation*}
\hackcenter{\begin{tikzpicture}[scale=.5]
    \draw[thick] (1,0) [out=90, in=-90] to (1,3);
    \draw[thick] (2,0) [out=90, in=-90] to (4,3);
    \draw[thick] (3,0) [out=90, in=-90] to (2,3);
    \draw[thick] (4,0) [out=90, in=-90] to (3,3);
    \draw[thick] (5,0) [out=90, in=-90] to (5,3);
    \filldraw[draw=black, fill=white] (.75,-.25) rectangle (2.25,.25);
    \filldraw[draw=black, fill=black] (2.75,-.25) rectangle (5.25,.25);
    \filldraw[draw=black, fill=black] (.75,2.75) rectangle (2.25,3.25);
    \filldraw[draw=black, fill=white] (2.75,2.75) rectangle (3.25,3.25);
    \filldraw[draw=black, fill=white] (3.75,2.75) rectangle (5.25,3.25);
    \node at (2.5,-1) {$1$};
\end{tikzpicture}}\qquad
\hackcenter{\begin{tikzpicture}[scale=.5]
    \draw[thick] (1,0) [out=90, in=-90] to (3,3);
    \draw[thick] (2,0) [out=90, in=-90] to (4,3);
    \draw[thick] (3,0) [out=90, in=-90] to (1,3);
    \draw[thick] (4,0) [out=90, in=-90] to (2,3);
    \draw[thick] (5,0) [out=90, in=-90] to (5,3);
    \filldraw[draw=black, fill=white] (.75,-.25) rectangle (2.25,.25);
    \filldraw[draw=black, fill=black] (2.75,-.25) rectangle (5.25,.25);
    \filldraw[draw=black, fill=black] (.75,2.75) rectangle (2.25,3.25);
    \filldraw[draw=black, fill=white] (2.75,2.75) rectangle (3.25,3.25);
    \filldraw[draw=black, fill=white] (3.75,2.75) rectangle (5.25,3.25);
    \node at (2.5,-1) {$-1$};
\end{tikzpicture}}\qquad
\hackcenter{\begin{tikzpicture}[scale=.5]
    \draw[thick] (1,0) [out=90, in=-90] to (4,3);
    \draw[thick] (2,0) [out=90, in=-90] to (5,3);
    \draw[thick] (3,0) [out=90, in=-90] to (1,3);
    \draw[thick] (4,0) [out=90, in=-90] to (2,3);
    \draw[thick] (5,0) [out=90, in=-90] to (3,3);
    \filldraw[draw=black, fill=white] (.75,-.25) rectangle (2.25,.25);
    \filldraw[draw=black, fill=black] (2.75,-.25) rectangle (5.25,.25);
    \filldraw[draw=black, fill=black] (.75,2.75) rectangle (2.25,3.25);
    \filldraw[draw=black, fill=white] (2.75,2.75) rectangle (3.25,3.25);
    \filldraw[draw=black, fill=white] (3.75,2.75) rectangle (5.25,3.25);
    \node at (2.5,-1) {$-1$};
\end{tikzpicture}}
\end{equation*}

\end{example}

\begin{example} Consider the product $(e_2h_2,e_2h_2)=-2$.  
The two contributing diagrams and their signs are given below. 

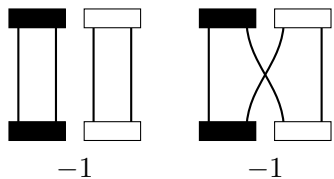
\begin{equation*}
\hackcenter{\begin{tikzpicture}[scale=.5]
    \draw[thick] (1,0) [out=90, in=-90] to (1,3);
    \draw[thick] (2,0) [out=90, in=-90] to (2,3);
    \draw[thick] (3,0) [out=90, in=-90] to (3,3);
    \draw[thick] (4,0) [out=90, in=-90] to (4,3);
    \filldraw[draw=black, fill=black] (.75,-.25) rectangle (2.25,.25);
    \filldraw[draw=black, fill=white] (2.75,-.25) rectangle (4.25,.25);
    \filldraw[draw=black, fill=black] (.75,2.75) rectangle (2.25,3.25);
    \filldraw[draw=black, fill=white] (2.75,2.75) rectangle (4.25,3.25);
    \node at (2.5,-1) {$-1$};
\end{tikzpicture}}\qquad
\hackcenter{\begin{tikzpicture}[scale=.5]
    \draw[thick] (1,0) [out=90, in=-90] to (1,3);
    \draw[thick] (2,0) [out=90, in=-90] to (3,3);
    \draw[thick] (3,0) [out=90, in=-90] to (2,3);
    \draw[thick] (4,0) [out=90, in=-90] to (4,3);
    \filldraw[draw=black, fill=black] (.75,-.25) rectangle (2.25,.25);
    \filldraw[draw=black, fill=white] (2.75,-.25) rectangle (4.25,.25);
    \filldraw[draw=black, fill=black] (.75,2.75) rectangle (2.25,3.25);
    \filldraw[draw=black, fill=white] (2.75,2.75) rectangle (4.25,3.25);
    \node at (2.5,-1) {$-1$};
\end{tikzpicture}}
\end{equation*}

\end{example}

\vspace{0.07in}

An equivalent description of the sign is as follows: for each pair of black platforms 
connected by $k$ strands, introduce the longest element of $S_k$ among those $k$ 
strands.  The sign of the diagram is then obtained just by counting crossings.  So if the 
diagram has all white platforms, a minimal coset representative is still used.  If it has 
all black platforms, a maximal coset representative diagram is used instead.  If 
platforms of both colors are used, the double coset representative chosen is neither 
minimal nor maximal in general, but is chosen as above. 

White and black boxes of size one represent the same element 
$h_1=e_1$ of $\Lambda$. 

\vspace{0.07in}

\begin{prop}\label{prop-e-relations} When $a+b$ is even,
\begin{equation}\label{eqn-e-relation-1}
h_ah_b=h_bh_a.
\end{equation}
When $a+b$ is odd,
\begin{equation}\label{eqn-e-relation-2}
h_ah_b+(-1)^ah_bh_a=(-1)^ah_{a+1}h_{b-1}+h_{b-1}h_{a+1}.
\end{equation}
\end{prop}
\noindent When $b=1$ and $a=2k$ is even, the odd degree relation takes the form
\begin{equation*}
h_1h_{2k}+h_{2k}h_1=2h_{2k+1}.
\end{equation*}
\begin{proof} We prove both relations simultaneously by induction on the total degree
 $a+b$.  First suppose $a+b$ is even.  It suffices to prove $(h_ah_b-h_bh_a,e_kx)=0$ 
for $k=1,2$ since the products $(h_ah_b,e_kx)$ and $(h_bh_a,e_kx)$ equal zero for 
any $k>2$ by Proposition \ref{prop-h}.  Computing graphically,
\begin{equation}\label{eqn-e-relation-proof}\begin{split}
&(h_ah_b,e_1x)=(h_{a-1}h_b,x)+(-1)^a(h_ah_{b-1},x),\\
&(h_ah_b,e_2x)=(-1)^{a-1}(h_{a-1}h_{b-1},x),
\end{split}\end{equation}
and likewise for $(h_bh_a,e_kx)$.  
When $k=1$ the difference $(h_ah_b-h_bh_a,e_1x)$ vanishes by the odd degree 
relation in degree $a+b-1,$ 
and when $k=2$ the difference vanishes by the even degree relation in degree 
$a+b-2$.

For $a+b$ odd, put together terms for $h_ah_b$, $(-1)^ah_bh_a$, 
$(-1)^ah_{a+1}h_{b-1}$, and $h_{b-1}h_{a+1}$ as in equation 
\eqref{eqn-e-relation-proof}.  The result vanishes with $k=1$ by some 
cancellation and the even degree relation in degree $a+b-1$, and with $k=2$ 
by the odd degree relation in degree $a+b-2$.\end{proof}
Since any element of $\sym$ is a linear combination of words in the $e_k$'s as 
well as a linear combination of words in the 
$h_k$'s, the same argument with these families of elements switched proves 
the following.
\begin{prop}\label{prop-h-relations} When $a+b$ is even,
\begin{equation}\label{eqn-h-relation-1}
e_ae_b=e_be_a.
\end{equation}
When $a+b$ is odd,
\begin{equation}\label{eqn-h-relation-2}
e_ae_b+(-1)^ae_be_a=(-1)^ae_{a+1}e_{b-1}+e_{b-1}e_{a+1}.
\end{equation}
\end{prop}
There are similar relations involving both $h$'s and $e$'s.
\begin{prop}\label{prop-e-h-relations} When $a+b$ is even,
\begin{equation}\label{eqn-e-h-relation-1}
h_ae_b=e_bh_a.
\end{equation}
When $a+b$ is odd,
\begin{equation}\label{eqn-e-h-relation-2}
h_ae_b+(-1)^ae_bh_a=(-1)^ah_{a+1}e_{b-1}+e_{b-1}h_{a+1}.
\end{equation}
\end{prop}
\begin{proof} The proof is along the same lines as that of Proposition \ref{prop-e-relations}, but with the slight complication that terms with $k>2$ do contribute.  As in that proof, we prove both relations simultaneously by induction on the total degree $a+b$.

For $a+b$ even, we compute
\begin{equation*}\begin{split}
&(h_ae_b,e_kx)=(h_{a-k}e_b+(-1)^{b-k+1}h_{a-k+1}e_{b-1},x),\\
&(e_bh_a,e_kx)=((-1)^{kb}e_bh_{a-k}+(-1)^{(k-1)(b-1)}e_{b-1}h_{a-k+1},x).
\end{split}\end{equation*}
We want to show that the difference of the two left arguments on the right-hand side is zero.  For $k$ even, this difference vanishes by applying \eqref{eqn-e-h-relation-1} twice in degree $a+b-k$.  For $k$ odd, it vanishes by applying \eqref{eqn-e-h-relation-2}.

For $a+b$ odd, we compute
\begin{equation*}\begin{split}
&(h_ae_b,h_kx)=(h_{a-k}e_b+(-1)^{a-k+1}h_{a-k+1}e_{b-1},x),\\
&(e_bh_a,h_kx)=((-1)^{kb}e_bh_{a-k}+(-1)^{(k-1)(b-1)}e_{b-1}h_{a-k+1},x),\\
&(h_{a+1}e_{b-1},h_kx)=(h_{a-k+1}e_{b-1}+(-1)^{a-k}h_{a-k+2}e_{b-2},x),\\
&(e_{b-1}h_{a+1},h_kx)=((-1)^{k(b+1)}e_{b-1}h_{a-k+1}+(-1)^{(k-1)b}e_{b-2}h_{a-k+2},x).
\end{split}\end{equation*}
Considering the linear combination $h_ae_b+(-1)^ae_bh_a$, the argument which is paired with $x$ is
\begin{equation}\label{eqn-larg}
h_{a-k}e_b-(-1)^{(k+1)b}e_bh_{a-k}+(-1)^{b-k}h_{a-k+1}e_{b-1}+(-1)^{k(b-1)}e_{b-1}h_{a-k+1}.
\end{equation}
For the linear combination $h_{a+1}e_{b-1}+(-1)^ae_{b-1}h_{a+1}$, we get
\begin{equation}\label{eqn-rarg}
(-1)^{k(b+1)}e_{b-1}h_{a-k+1}+(-1)^{b+1}h_{a-k+1}e_{b-1}+(-1)^{(k-1)b}e_{b-2}h_{a-k+2}+(-1)^kh_{a-k+2}e_{b-2}.
\end{equation}
For $k$ even, the difference of expressions \eqref{eqn-larg} and \eqref{eqn-rarg} vanishes by applying \eqref{eqn-e-h-relation-2} twice in degree $a+b-k$.  For $k$ odd, the difference vanishes by applying \eqref{eqn-e-h-relation-1} twice.\end{proof}

\begin{cor}\label{cor-bases} $\sym_n$ is a free $\Bbbk$-module, of which the 
families $\lbrace h_\lambda\rbrace_{\lambda\vdash n}$ and 
$\lbrace e_\lambda\rbrace_{\lambda\vdash n}$ are both bases.  Hence $\sym_n$ 
has the same graded rank as in the even ($q=1$) case, namely the number of 
partitions of $n$.\end{cor}
\begin{proof} Any element of $\sym$ is a linear combination of words in the $h$'s.  
By the relations \eqref{eqn-e-relation-1}, \eqref{eqn-e-relation-2}, only words 
whose subscripts are in non-increasing order are needed; that is, 
$\lbrace h_\lambda\rbrace_{\lambda\vdash n}$ is a spanning set.  Now let 
$\sym_\Z$ be $\sym$ considered over $\Bbbk=\Z$.  Since it is expressed as 
the quotient of a free $\Z$-module by the radical of a bilinear form, $\sym_\Z$ 
is itself a free $\Z$-module.  The mod 2 reduction of $h_\lambda$ coincides with 
the mod 2 reduction of the even ($q=1$) complete symmetric function 
$h_\lambda^\text{even}$, so the spanning set $\lbrace h_\lambda
\rbrace_{\lambda\vdash n}$ is linearly independent in $\sym_{\Z/2}$, hence in $\sym_\Z$.  The same 
argument works for the family $\lbrace e_\lambda\rbrace_{\lambda\vdash n}$.  
Now $\sym_\Z$ is a free $\Z$-module with the required bases, so 
$\sym=\sym_\Z\otimes_\Z\Bbbk$ is a free $\Bbbk$-module with the required 
bases.\end{proof}

\begin{cor}\label{cor-defining-relations} The algebra $\sym$ can be presented by generators $h_0=1,h_1,h_2\ldots$ subject to defining relations \eqref{eqn-e-relation-1} and \eqref{eqn-e-relation-2}.  There is an algebra automorphism of $\sym$ which takes $h_n$ to $e_n$ for all $n$.\end{cor}
\begin{proof} The first statement follows immediately from Proposition \ref{prop-e-relations} and Corollary \eqref{cor-bases}.  The second follows from Proposition \ref{prop-h-relations} and the fact that any word in the $h_n$'s can be expressed as a word in the $e_n$'s (by equation \eqref{eqn-defn-e}).\end{proof}
\noindent The automorphism $h_n\mapsto e_n$ is later denoted $\psi_1$.

\vspace{0.07in}

\noindent An alternate argument deduces that 
$\lbrace e_\lambda\rbrace_{\lambda\vdash n}$ is a basis of $(\sym_\Z)_n$ 
from the fact that $\lbrace h_\lambda\rbrace_{\vdash n}$ is and the following 
``semi-orthogonality'' property.  We call a minimal coset representative 
\textit{lite} if any two platforms in its diagram are connected by at most one strand.
\begin{prop}\label{prop-e-h-semi-orthogonality}
\begin{enumerate}
\item For all partitions $\lambda$,
\begin{equation}\label{eqn-e-h-semi-orthogonality-1}
(h_\lambda,e_{\lambda^T})=(-1)^{\ell(w_\lambda)},
\end{equation}
where $\ell(w_\lambda)$ is the Coxeter length of the unique lite minimal double coset 
representative $w_\lambda$ in $S_\lambda\backslash S_n/S_{\lambda^T}$.
\item If $\lambda$ is a partition and $\alpha$ is a composition with 
$\alpha>\lambda^T$ in the lexicographic order, then
\begin{equation}\label{eqn-e-h-semi-orthogonality-2}
(h_\lambda,e_\alpha)=(e_\lambda,h_\alpha)=0.
\end{equation}\end{enumerate}\end{prop}
A combinatorial description of $\ell(w_\lambda)$ is the number of strictly southwest-northeast pairs of boxes in the Young diagram corresponding to the partition $\lambda$.  One way to compute this quickly is to label each box of the Young diagram corresponding to $\lambda$ with the number of boxes which are strictly to the northeast.  Then $\ell(w_\lambda)$ is the sum of these numbers.  For example,
\begin{equation*}
\young(0000,3210,64,7)\qquad\ell(w_{4421})=22.
\end{equation*}

Before proving the proposition, we briefly recall the lexicographic order and the 
domainance partial order.  Let $\alpha=(a_1,\ldots,a_r)$, $\beta=(b_1,\ldots,b_s)$ 
be compositions of $n$.  We say $\alpha<\beta$ in the \textit{lexicographic order} 
(also called the \textit{dictionary order}) if $a_i<b_i$, where $i$ is the minimal index 
$j$ at which $a_j\neq b_j$.

Restricting our attention to partitions $\lambda=(\lambda_1,\ldots,\lambda_r)$, 
$\mu=(\mu_1,\ldots,\mu_s)$ of $n$, the lexicographic order refines the following 
partial order: we say $\lambda\leq\mu$ in the \textit{dominance partial order} if 
\begin{equation*}
\lambda_1+\lambda_2+\ldots+\lambda_i\leq\mu_1+\mu_2+\ldots+\mu_i\text{ for all }i=1,2,\ldots,n,
\end{equation*}
where we pad $\lambda,\mu$ by trailing zeroes when necessary.  The dominance partial order is a total order if and only if $n\leq5$ and is graded if and only if $n\leq 6$.  The lowest degree dominance-incomparable pairs are $\lbrace(3,1,1,1),(2,2,2)\rbrace$ and $\lbrace(4,1,1),(3,3)\rbrace$.

If all partitions of $n$ are listed lexicographically, it is not true that reversing the order swaps partitions whose corresponding Young diagrams are transposes of each other (this first occurs at $n=6$).  One can, however, refine the dominance partial order in such a way that this property holds.

\vspace{0.07in}
\begin{proof}[Proof of Proposition \ref{prop-e-h-semi-orthogonality}] The proof, except for the determination of the sign in equation \eqref{eqn-e-h-semi-orthogonality-1}, is exactly as in the case of classical symmetric functions.  To compute an inner product $(h_\lambda,e_\alpha)$, we must sign-count minimal double coset representative diagrams connecting a $\lambda$-arrangement of white platforms and a $\alpha$-arrangement of black platforms such that no pair of a black and a white diagram are connected by more than one strand.  For $\alpha=\lambda^T$, $\alpha_i$ equals the number of rows of $\lambda$ of size at least $i$; diagrammatically, the $i$-th black platform has exactly one strand going to each white platform of size at least $i$.  Hence there is a unique lite diagram connecting these two platform arrangements and it counts as $(-1)^{\ell(w)}$: in this diagram, the strands of the $i$-th black platform go to the first $\alpha_i$ white platforms, and these white platforms are precisely those white platforms which accept at least $i$ strands; and vice versa, switching black and white and switching $\lambda$ and $\alpha$.  For example, the unique lite diagram in computing $(h_{421},e_{3211})$ is:

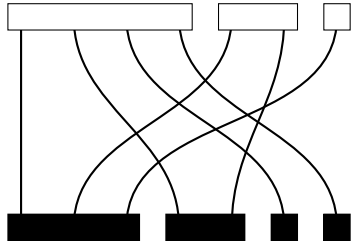
\begin{equation*}
\hackcenter{\begin{tikzpicture}[scale=.7]
    \draw[thick] (1,0) [out=90, in=-90] to (1,4);
    \draw[thick] (2,0) [out=90, in=-90] to (5,4);
    \draw[thick] (3,0) [out=90, in=-90] to (7,4);
    \draw[thick] (4,0) [out=90, in=-90] to (2,4);
    \draw[thick] (5,0) [out=90, in=-90] to (6,4);
    \draw[thick] (6,0) [out=90, in=-90] to (3,4);
    \draw[thick] (7,0) [out=90, in=-90] to (4,4);
    \filldraw[draw=black, fill=black] (.75,-.25) rectangle (3.25,.25);
    \filldraw[draw=black, fill=black] (3.75,-.25) rectangle (5.25,.25);
    \filldraw[draw=black, fill=black] (5.75,-.25) rectangle (6.25,.25);
    \filldraw[draw=black, fill=black] (6.75,-.25) rectangle (7.25,.25);
    \filldraw[draw=black, fill=white] (.75,3.75) rectangle (4.25,4.25);
    \filldraw[draw=black, fill=white] (4.75,3.75) rectangle (6.25,4.25);
    \filldraw[draw=black, fill=white] (6.75,3.75) rectangle (7.25,4.25);
\end{tikzpicture}}
\end{equation*}

For $\alpha>\lambda^T$, let $i$ be minimal such that $\alpha_i>\lambda^T_i$.  
Then filling a potential diagram as above, at the stage of connecting the $i$-th black 
platform, there are fewer than $\alpha_i$ white platforms which can still accept a 
new strand, so we are forced to send two strands from this black platform to the same 
white platform (pigeonhole principle).  So the diagram is zero.  For example, consider 
the next step in filling the unfinished diagram below for $\lambda=(4,2,1)$, 
$\alpha=(3,2,2)$:

\begin{equation*}
\hackcenter{\begin{tikzpicture}[scale=.7]
    \draw[thick] (1,0) [out=90, in=-90] to (1,4);
    \draw[thick] (2,0) [out=90, in=-90] to (5,4);
    \draw[thick] (3,0) [out=90, in=-90] to (7,4);
    \draw[thick] (4,0) [out=90, in=-90] to (2,4);
    \draw[thick] (5,0) [out=90, in=-90] to (6,4);
    \draw[thick] (6,0) [out=90, in=-90] to (3,4);
    \draw[thick] (7,0) [out=90, in=-90] to (7,.75);
    \draw[thick] (7,2.25) [out=90, in=-90] to (4,4);
    \filldraw[draw=black, fill=black] (.75,-.25) rectangle (3.25,.25);
    \filldraw[draw=black, fill=black] (3.75,-.25) rectangle (5.25,.25);
    \filldraw[draw=black, fill=black] (5.75,-.25) rectangle (7.25,.25);
    \filldraw[draw=black, fill=white] (.75,3.75) rectangle (4.25,4.25);
    \filldraw[draw=black, fill=white] (4.75,3.75) rectangle (6.25,4.25);
    \filldraw[draw=black, fill=white] (6.75,3.75) rectangle (7.25,4.25);
    \node at (7,1) {?};
    \node at (7,2) {?};
\end{tikzpicture}}
\end{equation*}

\noindent Connecting the strands marked ``?'' would result in a non-lite diagram.\end{proof}

For the remainder of this subsection, let $\Bbbk=\Z$.  Recall the representation 
theoretic interpretation of the even ($q=1$) analogue of Proposition 
\ref{prop-e-h-semi-orthogonality}: View $\sym$ as 
\begin{equation*}
K_0(S)=\bigoplus_{n\geq0}K_0(\C[S_n])
\end{equation*}
via the Frobenius characteristic 
map.  Since the algebra $\C[S_n]$ is semisimple, its usual and split Grothendieck groups are isomorphic; the simples and the indecomposables 
coincide, and they are all projective.  The multiplication and 
comultiplication in $K_0(S)$ come from induction and restriction between 
parabolic subgroups, and the bilinear form is determined by
\begin{equation*}
([V],[W])=\dim\Hom_{S_n}(V,W)
\end{equation*}
when $V,W$ are representations of $S_n$.  Under this identification,
\begin{equation*}\begin{split}
h_\lambda&\text{ corresponds to }[I_\lambda]=[\Ind_{S_\lambda}^{S_n}(L_{(n)})],\\
e_\lambda&\text{ corresponds to }[I_\lambda^-]=[\Ind_{S_\lambda}^{S_n}(L_{(1^n)})],
\end{split}\end{equation*}
where $|\lambda|=n$, $S_\lambda$ is the parabolic subgroup $S_{\lambda_1}\times\cdots\times S_{\lambda_r}\subseteq S_n$, and $L_\lambda$ is the irreducible representation of $S_n$ corresponding to the partition $\lambda$.  In particular, $L_{(n)}$ is the one-dimensional trivial representation and $L_{(1^n)}$ is the one-dimensional sign representation.  The $q=1$ analogue of the first statement of the Proposition (which is that $(h_\lambda,e_{\lambda^T})=1$) follows from the fact that the induced representations $I_\lambda$ and $I_{\lambda^T}^-$ share a unique irreducible summand, which occurs with multiplicity 1; this is the irreducible representation $L_\lambda$.  The $q=1$ analogue of the second statement follows from the absence of common irreducible summands in certain induced representations; this can be used to prove that all the $L_\lambda$'s constructed in this way are distinct.  We do not know an analogous representation theoretic interpretation (categorification) of the odd ($q=-1$) case.  As the sign in the first statement of the Proposition \ref{prop-e-h-semi-orthogonality} makes evident, dg-algebras or similar structures will likely be necessary to categorify $\sym$.

\vspace{0.07in}

Introduce the notation
\begin{align*}
&H_{\geq\lambda}=\text{span}_\Z\lbrace h_\mu:\mu\geq\lambda\rbrace,\\
&E_{\geq\lambda}=\text{span}_\Z\lbrace e_\mu:\mu\geq\lambda\rbrace
\end{align*}
and likewise with $\geq$ replaced by one of $\lbrace\leq,>,<\rbrace$ (lexicographic order), as subspaces of $\sym_n$ for $n=|\lambda|$.  

The following lemma will be useful in studying odd Schur functions.
\begin{lem}\label{lem-E-geq-nondegenerate} For any partition $\lambda\vdash n$, the bilinear form is nondegenerate when restricted to the subspaces $H_{\geq\lambda}$ and $E_{>\lambda^T}$ of $\sym_n$.\end{lem}
\begin{proof} By equation \eqref{eqn-e-h-semi-orthogonality-2} and nondegeneracy of the bilinear form, $(H_{\geq\lambda})^\perp=E_{>\lambda^T}$.  If $H_{\geq\lambda}\cap E_{>\lambda^T}=\lbrace0\rbrace$, then it follows that 
\begin{equation*}
\sym_n=H_{\geq\lambda}\oplus E_{>\lambda^T}
\end{equation*}
is an orthogonal decomposition, so that 
\begin{equation*}
\det(\cdot,\cdot)|_{H_{\geq\lambda}}\det(\cdot,\cdot)|_{E_{>\lambda^T}}=\det(\cdot,\cdot)=\pm1,
\end{equation*}
which implies that both factors on the left hand side are $\pm1$.  And since $H_{\geq\lambda}\cap E_{>\lambda^T}=\lbrace0\rbrace$ after reducing mod 2, the intersection must have been zero over $\Z$: Any nonzero element of $H_{\geq\lambda}\cap E_{>\lambda^T}$ which is zero mod 2 must be divisible by 2.  But then the result of dividing this element by 2 would also be in $H_{\geq\lambda}\cap E_{>\lambda^T}$.\end{proof} 

The previous lemma does not hold with $\geq$ replaced by $\leq$.  For instance, $(h_{11},h_{11})=0$.

%
\subsection{(Anti-)automorphisms, generating functions, and the antipode}\label{subsec-automorphisms}
%

If $\alpha$ is a composition, we write $\alpha^\text{rev}$ for the composition obtained by reverse-ordering $\alpha$.  We also use this notation for partitions $\lambda$ (though, of course, $\lambda^\text{rev}$ is rarely a partition).

We introduce three (anti-)automorphisms of $\sym$ which will be of use to us.  Their definitions and basic properties:
\begin{equation*}\begin{split}
&\psi_1(h_n)=e_n\\
&\qquad\text{(super-)bialgebra automorphism (not an involution),}\\
&\psi_2(h_n)=(-1)^{\binom{n+1}{2}}h_n\\
&\qquad\text{algebra involution (not a coalgebra homomorphism),}\\
&\psi_3(h_n)=h_n\\
&\qquad\text{superalgebra anti-involution (not a coalgebra homomorphism).}
\end{split}\end{equation*}
The meaning of the last being an anti-homomorphism of superalgebras is that
\begin{equation*}
\psi_3(xy)=(-1)^{\deg(x)\deg(y)}\psi_3(y)\psi_3(x)
\end{equation*}
for homogeneous elements $x,y$.  All three of these maps lift to $\sym'$ at $q=-1$.  By Proposition \ref{prop-h-relations} and Corollary \ref{cor-defining-relations}, $\psi_1$ is a well defined algebra automorphism.  Since the lifts of $\psi_2$ and $\psi_3$ to $\sym'$ preserve the defining relations \eqref{eqn-e-relation-1} and \eqref{eqn-e-relation-2}, $\psi_2$ and $\psi_3$ are themselves algebra automorphisms as well.

\vspace{0.07in}

Perhaps the only one of the $\psi_i$ whose introduction requires comment is $\psi_2$.  It is useful because the $h_k$'s and the $\psi_2(e_k)$'s satisfy a family of relations analogous to familiar ones in the even case:
\begin{equation}\label{eqn-gen-fn-consequence}
\sum_{k=0}^n(-1)^{k(n-k)}\psi_2(e_{n-k})h_k=0.
\end{equation}
In order to see the difference in meaning between this equation and equation \eqref{eqn-defn-e}, define the generating functions $H(t)=\sum_th_kt^k$ and $E(t)=\sum_te_kt^k$, with $t$ a variable of degree 1.  Then equation \eqref{eqn-gen-fn-consequence} is equivalent to the equation
\begin{equation}
\psi_2(E(t))H(t)=1
\end{equation}
holding in the ring $\sym[t]$.  The meaning here is that $t$ commutes (respectively skew commutes) with $h_k$ and $e_k$ for $k$ even (respectively odd); that is, we adjoin a super-central variable $t$ and extend $\psi_2$ to $\sym[t]$ by $\psi_2(t)=t$.  We do not know of a generating function interpretation of equation \eqref{eqn-defn-e}.

Another interpretation of $\psi_2$ is that there is a $\Z/2$-grading on $\sym$ determined by placing $h_n$ in degree 0 (respectively degree 1) if $n\equiv0,3$ (mod 4) (respectively $n\equiv1,2$ (mod 4)).  In terms of this grading, $\psi_2$ is the identity on the degree 0 part and minus the identity on the degree 1 part.

\vspace{0.07in}

By definition, $\psi_3(h_n)=h_n$.  However, $\psi_3(e_n)\neq e_n$ for $n>1$.  Nor is it true that $\psi_3(h_\lambda)=\pm h_\lambda$.  For instance,
\begin{equation*}
\psi_3(h_2h_1)=h_1h_2=2h_3-h_2h_1.
\end{equation*}
But $\psi_3$ does preserve $h_\lambda$ up to ``higher order terms.''
\begin{lem}\label{lemma-psi-3-e-h} $\psi_3(h_\lambda)$ is in $H_{\geq\lambda}$, and the coefficient of $h_\lambda$ in $\psi_3(h_\lambda)$ (when expanding in the complete basis) is $(-1)^{\binom{\lambda^T}{2}}$.\end{lem}
\begin{proof} Since $\psi_3$ is an superalgebra anti-homomorphism,
\begin{equation}\label{eqn-psi3-sign}
\psi_3(h_\lambda)=(-1)^{\sum_{i<j}\lambda_j\lambda_j}h_{\lambda^\text{rev}}.
\end{equation}
Now consider all compositions of a fixed degree to be ordered lexicographically.  By induction, then, it suffices to show that whenever $a<b$, $h_ah_b$ is in $H_{\geq(b,a)}$.  If $a+b$ is even, then $h_ah_b=h_bh_a$ and we are done.  If $a+b$ is odd, apply the odd degree $h$-relation:
\begin{equation*}
h_ah_b=(-1)^ah_bh_a+h_{b+1}h_{a-1}-(-1)^ah_{a-1}h_{b+1}.
\end{equation*}
The first and second terms on the right-hand side are now lexicographically greater than $h_ah_b$ and in non-increasing order, so it remains to express $h_{a-1}h_{b+1}$ as a linear combination of terms lexicographically higher.  To do so, apply the odd degree $h$-relation to $h_{a-1}h_{b+1}$, and then to the last term in that, and so forth until the left factor's subscript reaches one.  At this point apply $h_1h_{a+b-1}=2h_{a+b}-h_{a+b-1}h_1$, and we are done.  Going through the algorithm just described, $\sum_{i<j}(\lambda_i)\lambda_j$ minus signs are accrued (modulo $2$).  Combining this with the contribution of \eqref{eqn-psi3-sign}, the total (modulo $2$) number of minus signs is
\begin{equation*}
\sum_{i<j}\left(\lambda_i\lambda_j+(\lambda_i+1)\lambda_j\right)\equiv\binom{\lambda^T}{2}.
\end{equation*}\end{proof}

\vspace{0.07in}

The automorphism $\psi_1$ has infinite order, as
\begin{equation*}
\psi_1^m(h_2)=h_2-mh_1^2.
\end{equation*}
The anti-automorphism $\psi_1\psi_2\psi_3$ has infinite order as well, while $\psi_1\psi_2$ squares to the identity.

\begin{prop}\label{prop-psi432} Let $S=\psi_1\psi_2\psi_3$.  Then with the (co)multiplication and (co)unit already defined and with $S$ as antipode, $\sym$ has the structure of a non-involutory $\Z$-graded Hopf superalgebra, that is, a Hopf algebra object in the category $\Bbbk\text{-gmod}_{-1}$ with $S^2\neq1$.  The $\Z$-grading is compatible in the sense that the super-grading is just the mod 2 reduction of the $\Z$-grading.  We have
\begin{equation}\label{eqn-antipode-action}
\psi_1\psi_2(h_\lambda)=(-1)^{\binom{\lambda}{2}+|\lambda|}e_\lambda,	\qquad	\psi_1\psi_2(e_\lambda)=(-1)^{\binom{\lambda}{2}+|\lambda|}h_\lambda, \qquad
S(h_\lambda)=(-1)^{\binom{|\lambda|+1}{2}}e_{\lambda^\text{rev}}.
\end{equation}\end{prop}
\indent In the setting of ordinary vector spaces, if a bialgebra admits a Hopf antipode then this antipode is unique.  The same is true in more general settings including ours; see \cite{Ma}.  
\begin{proof} Letting $\eta$ be the unit and $\epsilon$ the counit, $(\eta\circ\epsilon)(h_\lambda)=\delta_{\lambda,(0)}$.  On a single $h_n$,
\begin{equation*}\begin{split}
(m\circ(S\otimes1)\circ\Delta)(h_n)&=(m\circ(\psi_1\psi_2\psi_3\otimes1))\left(\sum_{k=0}^nh_k\otimes h_{n-k}\right)\\
&=m\left(\sum_{k=0}^n(-1)^{\binom{k+1}{2}}e_k\otimes h_{n-k}\right)\\
&=\delta_{n,0}.
\end{split}\end{equation*}
If $x,y$ are homogeneous, then using Sweedler notation,
\begin{equation*}\begin{split}
(m\circ(S\otimes1)\circ\Delta)(xy)&=m\circ(S\otimes1)\left(\sum_{(x),(y)}(-1)^{\deg(x_{(2)})\deg(y_{(1)})}x_{(1)}y_{(1)}\otimes x_{(2)}y_{(2)}\right)\\
&=\sum_{(y)}(-1)^{\deg(x)\deg(y_{(1)})}S(y_{(1)})\left(\sum_{(x)}S(x_{(1)})x_{(2)}\right)y_{(2)}.
\end{split}\end{equation*}
Letting $x,y$ each be of the form $h_\lambda$, it follows by induction on $\ell(\lambda)$ that $(m\circ(S\otimes1)\circ\Delta)(h_\lambda)=\delta_{\lambda,(0)}$.  So $S$ is the Hopf antipode.

The expressions for $\psi_1\psi_2(h_\lambda)$ and $S(h_\lambda)$ in equation \eqref{eqn-antipode-action} are immediate from the definitions of $\psi_1,\psi_2,\psi_3$ and the above calculation.  In order to compute $\psi_1\psi_2(e_n)$, we proceed by induction (the $n=1$ case is clear).  Applying $\psi_3$ to equation \eqref{eqn-defn-e}, we have
\begin{equation}\label{eqn-dual-to-h-defn}
\sum_{k=0}^n(-1)^{\binom{k+1}{2}}h_{n-k}e_k=0.
\end{equation}
Hence
\begin{eqnarray}
(-1)^{\binom{n+1}{2}}\psi_1\psi_2(e_n)&\refequal{\eqref{eqn-defn-e}}&-\sum_{k=0}^{n-1}(-1)^{\binom{k+1}{2}}\psi_1\psi_2(e_kh_{n-k})\\
&=&-\sum_{k=0}^{n-1}(-1)^{\binom{n-k+1}{2}}h_ke_{n-k}\\
&\refequal{\eqref{eqn-dual-to-h-defn}}&h_n.
\end{eqnarray}
The second equality is by the inductive hypothesis.  Since $\psi_1\psi_2$ is a homomorphism, this immediately generalizes to prove the expression for $\psi_1\psi_2(e_\lambda)$ in equation \eqref{eqn-antipode-action}.
\end{proof}

\begin{cor} We have $\psi_2\psi_1\psi_2=\psi_1^{-1}$ and $\psi_1\psi_2\psi_1=\psi_2$.\end{cor}
\begin{proof} Immediate from the preceding proposition.\end{proof}

\noindent Let $\text{SAut}(\sym)$ be the $\Z/2$-graded group of algebra automorphisms and anti-automorphisms of $\sym$.  It follows from the above that the subgroup of $\text{SAut}(\sym)$ generated by $\psi_1,\psi_2$ is isomorphic to $\Z/2*\Z/2$.

%
\subsection{Relation to quantum quasi-symmetric functions}\label{subsec-comparison-qsymq}
%

The ring $Q\sym_q$ of \textit{quantum quasi-symmetric functions}, introduced in \cite{ThibonUng}, is a noncommutative deformation of Gessel's quasi-symmetric functions \cite{Gessel} whose definition uses Rosso's quantum shuffle product \cite{Rosso}.  There is a basis of $Q\sym_q$ known as the basis of \textit{ribbon Schur functions}.  In degree $n$, the ribbon Schur functions are indexed by compositions $\alpha$ of $n$ and are denoted $R_\alpha$.  

We now recall some combinatorial notions.  For a permutation $\sigma\in S_n$, we say a number $k\in\lbrace1,2,\ldots,n-1\rbrace$ is a \textit{descent} if $\sigma(k)>\sigma(k+1)$.  In terms of the strands diagrams of Subsection \ref{subsec-q-vector-spaces}, $k$ is a descent of $\sigma$ if and only if the $k$-th and $(k+1)$-st strands cross in a generic diagram for $\sigma$.  The associated \textit{descent composition} $C(\sigma)=(i_1,\ldots,i_r)$ is the composition of $n$ such that the set of descents of $\sigma$ is $\lbrace i_1,i_1+i_2,\ldots,i_1+\ldots+i_{r-1}\rbrace$.  In other words, the first $i_1$ strands do not cross, the next $i_2$ strands do not cross, and so forth; and strands $i_1$ and $i_1+1$ must cross, strands $i_1+i_2$ and $i_1+i_2+1$ must cross, and so forth.  The bilinear form on $Q\sym_q$ is given by
\begin{equation}\label{eqn-ribbon-schur-bilinear-form}
(R_\beta,R_\alpha)=\sum_{\substack{C(\sigma)=\alpha\\C(\sigma^{-1})=\beta}}q^{\ell(\sigma)}
\end{equation}
(sum over $\sigma\in S_n$, where $n=|\alpha|=|\beta|$).  This formula appears as equation (39) of \cite{ThibonUng} and as an unlabelled equation near the end of Section 10.15 of \cite{AguiarMahajan}.  In the latter reference, $Q\sym_q$ is defined in a more abstract manner, as the graded $q$-Hopf algebra associated via the bosonic Fock functor to the species of linear set compositions.  The dual species gives rise to a $q$-Hopf algebra $N\sym_q$ which is graded dual to $Q\sym_q$.  For $q$ not an algebraic integer, $Q\sym_q$ and $N\sym_q$ are isomorphic \cite[Proposition 12.38]{AguiarMahajan}.

Define elements $\widetilde{h}_\alpha$ of $\sym'$ by
\begin{equation}\label{eqn-h-tilde}
\widetilde{h}_\alpha=\sum_{\beta\leq\alpha}(-1)^{\ell(\alpha)-\ell(\beta)}h_\beta.
\end{equation}
In the above, for two compositions $\alpha,\beta$ of $n$, we say $\beta\leq\alpha$ if $\alpha$ refines $\beta$.  Note that the change of basis $h_\alpha\mapsto\widetilde{h}_\alpha$ is upper-triangular and unimodular.  Observe that, by equations \eqref{eqn-defn-e-alt} and \eqref{eqn-h-tilde},
\begin{equation}
e_n=(-1)^{\binom{n}{2}}\widetilde{h}_{(1^n)}.
\end{equation}
If we denote by $\lbrace\widehat{R}_\alpha\rbrace$ the basis of $N\sym_q$ dual to the basis $\lbrace R_\alpha\rbrace$ of $Q\sym_q$, is it easy to see that
\begin{equation*}\begin{split}
\sym'&\to N\sym_q\\
\widetilde{h}_\alpha&\mapsto\widehat{R}_\alpha
\end{split}\end{equation*}
defines an isomorphism between $\sym'$ and $N\sym_q$, since
\begin{equation}\label{eqn-h-tilde-bilinear-form}
(\widetilde{h}_\beta,\widetilde{h}_\alpha)=\sum_{\substack{C(\sigma)=\alpha\\C(\sigma^{-1})=\beta}}q^{\ell(\sigma)}.
\end{equation}

We can give a diagrammatic interpretation of equation \eqref{eqn-h-tilde-bilinear-form}.  Setting up platforms at top and bottom to compute $(\widetilde{h}_\beta,\widetilde{h}_\alpha)$ just as one would to compute $(h_\beta,h_\alpha)$, the following extra restriction is placed on diagrams: \textit{strands which start or end at adjacent positions not on the same platform must cross}.  The derivation of \eqref{eqn-h-tilde-bilinear-form} from \eqref{eqn-odd-bilform} is an easy exercise using inclusion-exclusion.

\begin{example} The computation of
\begin{equation*}
(h_{31},h_{22})=1+q^2
\end{equation*}
involves two admissible diagrams.  Only one of these diagrams, however, contributes to
\begin{equation*}
(\widetilde{h}_{31},\widetilde{h}_{22})=q^2.
\end{equation*}
\end{example}

%
\section{Other bases of $\sym$}
%

%
\subsection{Dual bases: odd monomial and forgotten symmetric functions}\label{subsec-dual-bases}
%

In the even ($q=1$) case, the dual bases to the elementary and complete symmetric functions are the forgotten and the monomial symmetric functions, respectively.  The monomial functions $\lbrace m_\lambda\rbrace$ get their name from the fact that when $\sym$ is viewed in terms of power series, they are sums of monomials of the same shape, 
\begin{equation}\label{eqn-m-power-series}
m_\lambda=\sum_\alpha x^\alpha.
\end{equation}
Here, $\lambda=(\lambda_1,\ldots,\lambda_r,0,\ldots)$ is a partition padded with infinitely many zeroes at the end, the sum ranges over all \textit{distinct} permutations $\alpha$ of $\lambda$, and $x^\alpha=x_1^{\alpha_1}x_2^{\alpha_2}\cdots$.  In terms of power series, no particularly nice description of the forgotten symmetric functions is known.  As a result they are often omitted from the discussion; whence their name.  From the point of view of self-adjoint Hopf algebras with a bilinear form (in the sense of \cite{ZelFinite}), however, they are just as natural a consideration as the monomial functions; see \cite[I.2]{Mac}.

\vspace{0.07in}

We now return to the odd ($q=-1$) case.  For $n\geq0$, define the \textit{odd monomial symmetric functions} $\lbrace m_\lambda\rbrace_{\lambda\vdash n}$ to be the dual basis to $\lbrace h_\lambda\rbrace_{\lambda\vdash n}$ and define the \textit{odd forgotten symmetric functions} $\lbrace f_\lambda\rbrace_{\lambda\vdash n}$ to be the dual basis to $\lbrace e_\lambda\rbrace_{\lambda\vdash n}$.  In other words, we define $m_\lambda$ and $f_\lambda$ by the conditions
\begin{equation*}
(h_\lambda,m_\mu)=\delta_{\lambda\mu},\qquad(e_\lambda,f_\mu)=\delta_{\lambda\mu}.
\end{equation*}
The monomial and forgotten functions through degree 4 are given in Subsection \ref{subsec-data-bases} of the Appendix.

\vspace{0.07in}

Define the coefficients $M_{\lambda\mu},M'_{\lambda\mu},M''_{\lambda\mu}$ (indexed over ordered pairs of partitions $\lambda,\mu$ of some $n$) by
\begin{equation}\label{eqn-M-defn}
M_{\lambda\mu}=(e_\lambda,h_\mu),\qquad M'_{\lambda\mu}=(h_\lambda,h_\mu),\qquad M''_{\lambda\mu}=(e_\lambda,e_\mu).
\end{equation}
The following change of basis relations are immediate consequences of \eqref{eqn-M-defn}:
\begin{equation}\begin{split}\label{eqn-hemf}
h_\lambda=\sum_{\mu\vdash n}M_{\lambda\mu}f_\mu,		\qquad	&h_\lambda=\sum_{\mu\vdash n}M'_{\lambda\mu}m_\mu,\\
e_\lambda=\sum_{\mu\vdash n}M_{\lambda\mu}m_\mu,	\qquad	&e_\lambda=\sum_{\mu\vdash n}M''_{\lambda\mu}f_\mu.
\end{split}\end{equation}
Along with the results of Subsection \ref{subsec-e-h}, we see that these change of basis matrices have the properties:
\begin{itemize}
\item $M_{\lambda\mu}$ is equal to 0 when $\mu>\lambda^T$ in the lexicographic order and equal to $\pm1$ when $\mu=\lambda^T$.  So the change of basis matrix is upper-left-triangular with $\pm1$'s on the diagonal.
\item The matrix for $M'_{\lambda\mu}$ is symmetric and has determinant equal to $\pm1$.
\item The matrix for $M''_{\lambda\mu}$ is symmetric and has determinant equal to $\pm1$.
\end{itemize}
Since $(e_\lambda,e_\mu)=(h_\lambda,h_\mu)$ when $q=1$, the $q=1$ analogues of $M'_{\lambda\mu}$ and $M''_{\lambda\mu}$ are equal.  Their combinatorial interpretations in Proposition \ref{prop-dual-basis-combinatorics} below are the same when the signs are omitted.  In the odd ($q=-1$) case, they differ because $(e_k,e_k)=(-1)^{\binom{k}{2}}$, as this sign comes up whenever $k$ strands connect the same two black platforms.

\vspace{0.07in}

The determinant of the matrix $M$ is not hard to compute.  $M$ is upper-left-triangular by Proposition \ref{prop-e-h-semi-orthogonality}, and the anti-diagonal entry $(h_\lambda,e_{\lambda^T})$ equals $(-1)^{\ell(w_\lambda)}$.  But this entry and the anti-diagonal entry $(h_{\lambda^T},e_\lambda)$ are equal, so the determinant of $M$ in degree $n$ is a sign computed only from self-transpose diagrams:
\begin{equation}
\det(M_n)=\prod_{\lambda=\lambda^T}(-1)^{\ell(w_{\lambda})}.
\end{equation}
The determinants $\det(M'_n)$ and $\det(M''_n)$ both equal $\det(M_n)$ times the determinant of the change of basis between the $e$- and $h$-bases.  Note that self-transpose Young diagrams with $n$ boxes are in a natural bijection with partitions of $n$ into distinct odd positive integers.  Under this bijection, the sign $(-1)^{\ell(w_\lambda)}$ has a factor of $-1$ for each summand which is congruent to 3 modulo 4.

\vspace{0.07in}

The proof of the following proposition, which we omit, is essentially the same as in the even ($q=1$) case, but with the extra bookkeeping of signs.  For the even case; see Proposition 37.5 of \cite{Bump}.  For a matrix $A$, define the composition $\mathrm{row}(A)$ (respectively $\mathrm{col}(A)$) to consist of the row (respectively column) sums of $A$.  If $A$ is an $\N$-matrix (that is, its entries are all natural numbers), there are two sorts of signs we can attach to $A$ when counting matrices.  Our natural numbers include zero: $\N=\lbrace0,1,2,\ldots\rbrace$.
\begin{itemize}
\item To count \textit{SW-NE pairs} means to accrue a sign of $(-1)^{ab}$ for every pair of entries in which an $a$ is strictly below and strictly to the left of a $b$.
\item To count \textit{cables} means to accrue a sign of $(-1)^{\binom{a}{2}}$ for every entry $a$.  Since 0- and 1-cables accrue $1$'s, this is not interesting for $\lbrace0,1\rbrace$-matrices.
\end{itemize}
\begin{prop}\label{prop-dual-basis-combinatorics} The numbers defined in equation \eqref{eqn-M-defn} have the following combinatorial interpretations:
\begin{enumerate}
\item $M_{\lambda\mu}$ equals the signed count of $\lbrace0,1\rbrace$-matrices $A$ with $\mathrm{row}(A)=\lambda$ and $\mathrm{col}(A)=\mu$.  The sign counts SW-NE pairs.
\item $M'_{\lambda\mu}$ equals the signed count of $\N$-matrices $A$ with $\mathrm{row}(A)=\lambda$ and $\mathrm{col}(A)=\mu$.  The sign counts SW-NE pairs.
\item $M''_{\lambda\mu}$ equals the signed count of $\N$-matrices $A$ with $\mathrm{row}(A)=\lambda$ and $\mathrm{col}(A)=\mu$.  The sign counts SW-NE pairs and cables.
\end{enumerate}\end{prop}

\begin{example} We compute the $(3,2),(2,2,1)$ entry of the matrices $M,M',M''$.  Below are the five $\N$-matrices with row sum $(3,2)$ and column sum $(2,2,1)$, and their contributions to $M,M',M''$.
\begin{center}\begin{tabular}{| c | c c c |}
\hline
matrix								&	contribution to $M$		&	contribution to $M'$		&	contribution to $M''$						\\
\hline
$\begin{pmatrix}2&1&0\\0&1&1\end{pmatrix}$	&	0		&	$(-1)^0$	&	$(-1)^0(-1)^{\binom{3}{2}}$			\\
$\begin{pmatrix}2&0&1\\0&2&0\end{pmatrix}$	&	0		&	$(-1)^2$	&	$(-1)^2(-1)^{\binom{3}{2}+\binom{3}{2}}$	\\
$\begin{pmatrix}1&2&0\\1&0&1\end{pmatrix}$	&	0		&	$(-1)^2$	&	$(-1)^2(-1)^{\binom{3}{2}}$			\\
$\begin{pmatrix}1&1&1\\1&1&0\end{pmatrix}$	&	$(-1)^3$	&	$(-1)^3$	&	$(-1)^3(-1)^0$				\\
$\begin{pmatrix}0&2&1\\2&0&0\end{pmatrix}$	&	0		&	$(-1)^6$	&	$(-1)^6(-1)^{\binom{3}{2}+\binom{3}{2}}$	\\
\hline
\end{tabular}\end{center}
Therefore
\begin{equation*}
M_{(3,2),(2,2,1)}=-1,\qquad	M_{(3,2),(2,2,1)}'=3, \qquad	M_{(3,2),(2,2,1)}''=-1.
\end{equation*}
\end{example}

\vspace{0.07in}

We end this section by pointing out that the above results are enough to compute the matrix of the bilinear form in any of the bases described so far.  For instance, since $M'$ is the matrix of the bilinear form in the $h$-basis, $M$ is the matrix which takes the $f$-basis to the $h$-basis, and $M=M^T$, the matrix $M^{-1}M'M^{-1}$ is the matrix of the bilinear form in the $f$-basis.

%
\subsection{Primitive elements, odd power symmetric functions, and the center of $\sym$}\label{subsec-primitives}
%

For this section assume $\Bbbk$ is a field of characteristic zero.

\vspace{0.07in}

Recall that an element $x$ of a ($q$-)Hopf algebra is called \textit{primitive} if
\begin{equation*}
\Delta(x)=1\otimes x+x\otimes 1.
\end{equation*}
In the even ($q=1$) case, the primitive elements of $\sym$ are spanned by the power sum functions
\begin{equation*}
p_n=\sum_jx_j^n.
\end{equation*}
In the odd setting, however, there are only ``half'' as many.
\begin{prop} The subspace of primitive elements $P$ in $\sym$ is spanned by the elements $m_1$ and $m_{2k}$ for $k\geq1$.\end{prop}
\begin{proof} Let $I\subset\sym$ be the ideal generated by all elements of positive degree.  By the general theory of self-adjoint connected graded Hopf algebras with a bilinear form, $P=(I^2)^\perp$ (Lemma 1.7 of \cite{ZelFinite}).  It is clear from the $h$-relations (Propositon \ref{prop-e-relations}) or the $e$-relations (Proposition \ref{prop-h-relations}) that in each degree $n$,
\begin{equation*}
I^2\cap\sym_n=\begin{cases}0&n=1,\\
\text{span}_\Bbbk\lbrace h_\lambda:\lambda\neq(n)\rbrace&n\text{ is even},\\
\sym_n&n\text{ is odd and}\geq3.\end{cases}
\end{equation*}
The result follows.\end{proof}

Note that $f_n=\pm m_n$, so the $f_{2k}$ are primitive as well (the sign is the same as the coefficient of $h_n$ in the expansion of $e_n$ in the $h$-basis).  We define, therefore, the $n$-th \textit{odd power symmetric function} to be 
\begin{equation*}
p_n=m_n.
\end{equation*}
The first few $p_n$ are:
\begin{equation*}\begin{split}
&p_1=h_1,\\
&p_2=h_{11},\\
&p_3=h_{111}+h_{21}-h_3,\\
&p_4=-h_{1111}-2h_{22}+4h_4,\\
&p_5=h_{11111}+h_{2111}+3h_{221}-h_{311}-3h_{32}-9h_{41}+9h_5,\\
&p_6=h_{111111}+3h_{2211}-3h_{33}-6h_{411}+6h_{51}.
\end{split}\end{equation*}

\vspace{0.07in}

\begin{prop} The element $p_k$ belongs to the center of $\sym$ if and only if $k$ is even.\end{prop}
\begin{proof} We will show that  $(p_kh_m,e_\lambda)=(h_mp_k,e_\lambda)$ for every $m\geq0$ and every $\lambda\vdash(k+m)$, if and only if $k$ is even.  Let $\ell$ be the length of $\lambda$.  The coproduct of $e_\lambda$ is
\begin{equation*}
\Delta(e_\lambda)=\prod_{i=1}^\ell\sum_{j=0}^{\lambda_i}e_j\otimes e_{\lambda_i-j}=\sum_{\alpha}(e_{a_1}\otimes e_{\lambda_1-a_1})\cdots(e_{a_\ell}\otimes e_{\lambda_\ell-a_\ell}),
\end{equation*}
where the last sum is over all $\alpha$ such that $|\alpha|=k+m$ and $0\leq a_j\leq\lambda_j$ for each $j$.  When paired against $p_kh_m$ or $h_mp_k$, only partitions $\lambda=(k+1,1^{m-1})$ and $\lambda=(k,1^m)$ yield nonzero results.  It is straightforward to check that
\begin{eqnarray*}
&(p_kh_m,e_{k+1}e_1^{m-1})=1,		\qquad	&(h_mp_k,e_{k+1}e_1^{m-1})=(-1)^{k(m-1)},\\
&(p_kh_m,e_ke_1^m)=1,			\qquad	&(h_mp_k,e_ke_1^m)=(-1)^{km},
\end{eqnarray*}
using the adjointness of multiplication and comultiplication.  The result follows.\end{proof}
In fact, the center (which coincides with the supercenter) is precisely the polynomial algebra generated by the $p_{2k}$'s.  This will follow from the results of \cite{EKL}, but it would be nice to have a proof wholly within the framework of the Hopf algebra approach.

%
\subsection{Odd Schur functions}\label{subsec-schur}
%

We begin by reviewing some terminology from the combinatorics of Young diagrams.  Let $\lambda$ be a Young diagram.  A \textit{Young tableau} $T$ of shape $\lambda$ is an assignment of a positive integer to each box of $\lambda$.  We say $T$ is \textit{semistandard} if its entries are non-decreasing in all rows and strictly increasing in all columns.  

The \textit{content} of a semistandard Young tableau $T$ of shape $\lambda$, denoted $\cont(T)$, is the weak composition $\alpha=(a_1,\ldots,a_r)$ of $|\lambda|$ defined by
\begin{equation*}
a_i=\text{the number of entries of }T\text{ equal to }i.
\end{equation*}
In the even ($q=1$) case, the Schur functions $\lbrace s_\lambda\rbrace_{\lambda\vdash n}$ form an orthonormal basis of $\sym_n$.  In terms of power series, they are generating functions for semistandard Young tableaux.  If we denote the set of semistandard Young tableaux of shape $\lambda\vdash n$ by $\text{SSYT}(\lambda)$, one definition of the Schur function $s_\lambda$ is
\begin{equation*}
s_\lambda=\sum_{T\in\text{SSYT}(\lambda)}x^{\cont(T)}.
\end{equation*}
Then if one defines the \textit{Kostka number} associated to partitions $\lambda,\mu$ to be 
\begin{equation*}
K_{\lambda\mu}=\text{the number of semistandard Young tableaux of shape }\lambda\text{ and content }\mu,
\end{equation*}
it follows from the above and \eqref{eqn-m-power-series} that
\begin{equation}\label{eqn-even-s-m}
s_\lambda=\sum_{\mu\vdash n}K_{\lambda\mu}m_\mu.
\end{equation}
Having expressed Schur functions in terms of the dual basis to the complete functions, we have a definition which we can attempt to mimic in the odd ($q=-1$) case.  An essential feature in the even case is that the Schur functions $\lbrace s_\lambda\rbrace_{\lambda\vdash n}$ form an orthonormal basis of $\sym_n$; in the odd case, Schur functions will be orthogonal but their norms may be either 1 or $-1$.

\vspace{0.07in}

In the odd case, we define the \textit{odd Schur functions} by a change of basis relation closely related to \eqref{eqn-even-s-m},
\begin{equation}\label{eqn-s-to-h}
h_\mu=\sum_{\lambda\vdash n}K_{\lambda\mu}s_\lambda.
\end{equation}
To define the coefficients $K_{\lambda\mu}$, the \textit{odd Kostka numbers}, we first define the sign associated to a Young tableau $T$.  For a Young tableau $T$, let $w_r(T)$ be its \textit{row word}, that is, the string of numbers obtained by reading the entries of $T$ row by row from left to right, bottom to top.  Then define $\sign(T)$ to be the sign of the minimal length permutation which sorts $w_r(T)$ into non-decreasing order.  The sign of a standard Young tableau was introduced by Stanley in \cite{StanleySignImbalance}; our notion is an obvious generalization.

For a Young diagram $\lambda$, let $T_\lambda$ be the unique semistandard Young tableau with shape and content both equal to $\lambda$.  In other words every first-row entry of $T_\lambda$ is a 1, every second-row entry is a 2, and so forth.  With these notations established, we can define
\begin{equation}
K_{\lambda\mu}=\sign(T_\lambda)\sum_T\sign(T),
\end{equation}
where the sum is over all semistandard Young tableaux $T$ of shape $\lambda$ and content $\mu$.  Note that $K_{(n)\mu}=1$ for all $\mu\vdash n$, $K_{(1^n)\mu}=\delta_{\mu,(1^n)}$, and $K_{\lambda\lambda}=1$ for all $\lambda$.  Odd Kostka numbers of the form $K_{\lambda(1^n)}$ sign-count standard Young tableaux of a given shape; up to an overall sign, this count is what Stanley calls the \textit{sign imbalance} of the partition $\lambda$ \cite{StanleySignImbalance}.  Sign imbalance is a topic of current study; see \cite{Kim}, \cite{Lam2}.  Tables of odd Kostka numbers are included in Subsection \ref{subsec-data-bases} of the Appendix.

\vspace{0.07in}

\begin{example}
\noindent To compute $K_{(2,2,1),(1^5)}$:\\
\begin{center}
\begin{tabular}{| c c c c c c |}
\hline
&&&&&\\
tableau $T$	&	$\young(12,34,5)$	&	$\young(12,35,4)$	&	$\young(13,24,5)$	&	$\young(13,25,4)$	&	$\young(14,25,3)$	\\
$\sign(T)$		&	$+$				&	$-$				&	$-$				&	$+$				&			$-$				\\
&&&&&\\
\hline
\end{tabular}\\
\end{center}
Since $\sign(T_{(2,2,1)})=1$, we have $K_{(2,2,1),(1^5)}=-1$.
\end{example}

\begin{example}
\noindent To compute $K_{(3,1,1),(2,1,1,1)}$:\\
\begin{center}
\begin{tabular}{| c c c c |}
\hline
&&&\\
tableau $T$	&	$\young(112,3,4)$	&	$\young(113,2,4)$	&	$\young(114,2,3)$	\\
$\sign(T)$		&	$-$				&	$+$				&	$-$				\\
&&&\\
\hline
\end{tabular}\\
\end{center}
Since $\sign(T_{(3,1,1)})=-1$, we have $K_{(3,1,1),(2,1,1,1)}=1$.
\end{example}

\vspace{0.07in}

Subsection \ref{subsec-data-bases} of the Appendix lists the odd Schur functions through degree 5 in the complete functions basis.  The following theorem of Reifegerste and Sj\"{o}strand refines the usual RSK correspondence (see \cite{Fulton}) by keeping track of signs.  We give a new proof in Section \ref{sec-rsk}.
\begin{thm}[Odd RSK Correspondence I \cite{Reifegerste}, \cite{Sjostrand}]\label{thm-RSK-I} The RSK map is a bijection
\begin{equation}
\text{RSK}:\lbrace\quad
\substack{
\N\text{-matrices }A\text{ with} \\
\row(A)=\mu\text{ and }\col(A)=\rho
}\quad\rbrace\rightarrow\lbrace
\substack{\quad
\text{pairs }(P,Q)\text{ of semistandard Young tableaux of the} \\
\text{same shape, with }\cont(P)=\mu\text{ and }\cont(Q)=\rho}\quad\rbrace,
\end{equation}
under which the sign of $A$ as in the computation of $M_{\mu\rho}'$ equals $(-1)^{\binom{\lambda^T}{2}}\sign(P)\sign(Q)$, where $\lambda=\text{shape}(P)=\text{shape}(Q)$.  In particular,
\begin{equation}\label{eqn-RSK-I}\begin{split}
M_{\mu\rho}'&=\sum_{\lambda\vdash n}(-1)^{\binom{\lambda^T}{2}}K_{\lambda\mu}K_{\lambda\rho}\\
&=\sum_{\lambda\vdash n}(-1)^{\lambda_2+\lambda_4+\lambda_6+\ldots}K_{\lambda\mu}K_{\lambda\rho}.
\end{split}\end{equation}
\end{thm}

\begin{cor}\label{cor-m-to-s} The Schur function $s_\lambda$ can be expressed in terms of the monomial functions as
\begin{equation}\label{eqn-m-to-s}
(-1)^{\binom{\lambda^T}{2}}s_\lambda=\sum_{\mu\vdash n}K_{\lambda\mu}m_\mu.
\end{equation}\end{cor}
\begin{proof} Define matrices $A,B,C$, all square and indexed by all partitions of $n$, by
\begin{equation*}\begin{split}
&A_{\lambda\mu}=M_{\lambda\mu}',\\
&B_{\lambda\mu}=K_{\lambda\mu},\\
&C_{\lambda\mu}=(-1)^{\binom{\lambda^T}{2}}K_{\lambda\mu}.
\end{split}\end{equation*}
The ordering on the index set can be taken to be any total ordering which refines the dominance partial order.  In these terms, equation \eqref{eqn-s-to-h} says that $B^T$ takes the Schur basis to the complete basis and equation \eqref{eqn-hemf} says that $A$ takes the monomial basis to the complete basis.  It follows that $(B^T)^{-1}A$ takes the monomial basis to the Schur basis.  Now equation \eqref{eqn-RSK-I} says that $A=B^TC$, proving the corollary.\end{proof}

\begin{cor}\label{cor-schur-orthonormality} The Schur functions are signed-orthonormal:
\begin{equation}\label{eqn-schur-orthonormality}
(s_\lambda,s_\mu)=(-1)^{\binom{\lambda^T}{2}}\delta_{\lambda,\mu}.
\end{equation}\end{cor}
\begin{proof}
In $\sym\otimes\sym$, equations \eqref{eqn-m-to-s} and \eqref{eqn-s-to-h} imply
\begin{equation*}
\sum_{\lambda\vdash n}(-1)^{\binom{\lambda^T}{2}}s_\lambda\otimes s_\lambda=\sum_{\lambda,\mu\vdash n}K_{\lambda\mu}m_\mu\otimes s_\lambda=\sum_{\mu\vdash n}m_\mu\otimes h_\mu.
\end{equation*}
Since $\lbrace m_\lambda\rbrace_{\lambda\vdash n}$ and $\lbrace h_\lambda\rbrace_{\lambda\vdash n}$ are dual bases, it follows that $\lbrace s_\lambda\rbrace_{\lambda\vdash n}$ and $\lbrace(-1)^{\binom{\lambda^T}{2}}s_\lambda\rbrace_{\lambda\vdash n}$ are dual bases.\end{proof}

In order to express the Schur functions in the elementary and forgotten bases, note that the two following properties uniquely characterize the Schur functions:
\begin{enumerate}
\item $(s_\lambda,h_\mu)=0$ if $\mu>\lambda$ (lexicographic order).
\item For certain integers $a_\mu$ (depending on $\lambda$), 
\begin{equation}\label{eqn-h-to-s}
s_\lambda=h_\lambda+\sum_{\mu>\lambda}a_\mu h_\mu.
\end{equation}
\end{enumerate}
That these uniquely determine the Schur functions follows from Lemma \ref{lem-E-geq-nondegenerate}.  The first property follows immediately from equation \eqref{eqn-m-to-s} and the second follows from equation \eqref{eqn-s-to-h}.  We think of these conditions as an inductive definition of $s_\lambda$, starting from $s_{(n)}=h_n$.
\begin{prop}\label{prop-s-prime} Define the elements $s_\lambda'$ of $\sym$ inductively as follows: $s_{(1^n)}'=e_n$, and the following two properties hold:
\begin{enumerate}
\item $(s_\lambda',e_\mu)=0$ if $\mu>\lambda^T$ (lexicographic order).
\item For certain integers $b_\mu$ (depending on $\lambda$),
\begin{equation*}
s_\lambda'=e_{\lambda^T}+\sum_{\mu>\lambda^T}b_\mu e_\mu.
\end{equation*}
\end{enumerate}
Then $s_\lambda'=(-1)^{\ell(w_\lambda)+\binom{\lambda^T}{2}}s_\lambda$.\end{prop}
\begin{proof} By Lemma \ref{lem-E-geq-nondegenerate} and the property (2) preceding the statement of the Proposition, the space $(H_{\geq\lambda}\cap E_{\geq\lambda^T})\otimes\Q$ is one-dimensional and spanned by $s_\lambda$.  But it is also spanned by $s_\lambda'$, by property (2) in the statement of the proposition.  In order to determine the constant by which they differ, we compute
\begin{equation*}\begin{split}
(s_\lambda,s_\lambda')&=(h_\lambda,e_{\lambda^T})+\sum_{\mu>\lambda}a_\mu(h_\mu,e_{\lambda^T})\\
&\qquad\qquad+\sum_{\mu>\lambda^T}b_\mu(h_\lambda,e_\mu)+\sum_{\substack{\rho>\lambda\\\mu>\lambda^T}}a_\rho b_\mu(h_\rho,e_\mu)\\
&\refequal{\eqref{eqn-e-h-semi-orthogonality-2}}(-1)^{\ell(w_\lambda)}.
\end{split}\end{equation*}
Hence, by the signed orthonormality of Schur functions, $s_\lambda'=(-1)^{\ell(w_\lambda)+\binom{\lambda^T}{2}}s_\lambda$.\end{proof}

\begin{lem} The involution $\psi_1\psi_2$ acts on Schur functions as follows:
\begin{equation}
\psi_1\psi_2(s_\lambda)=(-1)^{\ell(w_\lambda)+|\lambda|}s_{\lambda^T}.
\end{equation}\end{lem}
\begin{proof}
We express $s_\lambda$ in terms of both complete and elementary functions and then compare the results ($a_\mu,b_\mu$ are integers depending on $\lambda$ and on $\mu$; their particular values are immaterial):
\begin{equation*}\begin{split}
\psi_1\psi_2(s_\lambda)&=\psi_1\psi_2\left(h_\lambda+\sum_{\mu>\lambda}a_\mu h_\mu\right)\\
&=(-1)^{\binom{\lambda}{2}+|\lambda|}e_\lambda+\sum_{\mu>\lambda}(-1)^{\binom{\mu}{2}+|\mu|}a_\mu h_\mu\\
\psi_1\psi_2(s_\lambda)&=\psi_1\psi_2\left((-1)^{\ell(w_\lambda)+\binom{\lambda^T}{2}}e_{\lambda^T}+\sum_{\mu>\lambda^T}b_\mu e_\mu\right)\\
&=(-1)^{\ell(w_\lambda)+|\lambda|}h_{\lambda^T}+\sum_{\mu>\lambda^T}(-1)^{\binom{\mu}{2}+|\mu|}b_\mu h_\mu.
\end{split}\end{equation*}
Since $H_{\geq\lambda}\cap E_{\geq\lambda^T}$ is generated by $s_{\lambda^T}$ as in the proof of Proposition \ref{prop-s-prime}, it follows that both the above expressions for $\psi_1\psi_2(s_\lambda)$ are equal to plus or minus $s_{\lambda^T}$.  Considering the leading coefficient of either one, we see that the sign between $\psi_1\psi_2(s_\lambda)$ and $s_{\lambda^T}$ must be $(-1)^{\ell(w_\lambda)+|\lambda|}$.\end{proof}

\begin{cor} The Schur function basis is related to the monomial and the complete bases as follows:
\begin{equation}\begin{split}
&(-1)^{\binom{\mu}{2}+|\mu|}e_\mu=\sum_{\lambda\vdash n}(-1)^{\ell(w_\lambda)+|\lambda|}K_{\lambda^T\mu}s_\lambda,\\
&(-1)^{\ell(w_\lambda)+\binom{\lambda^T}{2}+|\lambda|}s_\lambda=\sum_{\mu\vdash n}(-1)^{\binom{\mu}{2}+|\mu|}K_{\lambda^T\mu}f_\mu.
\end{split}\end{equation}\end{cor}
\begin{proof} Apply $\psi_1\psi_2$ to equations \eqref{eqn-s-to-h} and \eqref{eqn-m-to-s}.\end{proof}

\begin{cor}[``Odd RSK Correspondence II'']\label{cor-RSK-II} The following formula holds:
\begin{equation}\label{eqn-RSK-II}\begin{split}
(-1)^{\binom{\mu}{2}+|\mu|+\binom{\rho}{2}+|\rho|}M''_{\mu\rho}&=\sum_{\lambda\vdash n}(-1)^{\binom{\lambda^T}{2}}K_{\lambda^T\mu}K_{\lambda^T\rho}\\
&=\sum_{\lambda\vdash n}(-1)^{\lambda_2+\lambda_4+\lambda_6+\ldots}K_{\lambda^T\mu}K_{\lambda^T\rho}.
\end{split}\end{equation}\end{cor}
\begin{proof} Argue as in the proof of Corollary \ref{cor-m-to-s}.\end{proof}
\noindent Why the scare quotes around the name of the corollary?  Unlike the formula for $M_{\mu\rho}'$ (Odd RSK Correspondence I, Theorem \ref{thm-RSK-I}), it does not appear that the above formula can be refined to a matching of signs between particular matrices and their RSK-corresponding pairs of semistandard Young tableaux.  Such a refined correspondence is possible after permuting the matrices counted in a particular $M''_{\mu\rho}$, but we do not know of a general rule governing these permutations.

\vspace{0.07in}

We conclude this section by comparing odd Schur functions to a generalization due to Lascoux, Leclerc, and Thibon \cite{LLT2} of the spin-weight domino symmetric functions of Carr\'{e} and Leclerc \cite{CarreLeclerc}.  We will review the definition of these functions very briefly; see \cite{LLT2} for details and examples.  Let $\lambda$ be a Young diagram.  A \textit{domino tableau} $D$ of shape $\lambda$ is a tiling of $\lambda$ by $2\times1$ dominos and an assignment of a positive integer to each domino which is weakly increasing in rows and strictly increasing in columns (in order for $\lambda$ to admit a domino tiling, it is necessary but not sufficient that $|\lambda|$ be even).  In the polynomial ring $\Z[x_1,x_2,\ldots,x_n]$, let $x^D$ be the monomial $\prod_ix_i^{a_i}$, where $a_i$ is the number of dominos of $D$ labelled $i$.  The \textit{spin} of $D$, denoted $s(D)$, is defined to be one-half the number of vercally oriented dominos in $D$.  If $\lambda$ has at least one domino tiling, let $s^*(\lambda)$ be the highest possible spin of a domino tiling of $\lambda$ and define the \textit{cospin} of a domino tableau $D$ of shape $\lambda$ to be $\widetilde{s}(D)=s^*(\lambda)-s(D)$.  The spin $s(D)$ is a non-negative half-integer and the cospin $\widetilde{s}(D)$ is always an integer.

For an indeterminate $q$, let
\begin{equation}
G_\lambda=\sum_Dq^{\widetilde{s}(D)}x^D,
\end{equation}
the sum being over all domino tableaux $D$ of shape $\lambda$.  This is called the \textit{spin-weight domino symmetric function} corresponding to $\lambda$; it is an (ordinary, not odd) symmetric function.  It was pointed out to us by Thomas Lam that there is likely a connection between spin-weight domino functions and odd Schur functions, but we do not know a precise relationship.  Two pieces of evidence for such a connection are that both are closely related to the combinatorics of tableau signs and that the norm of an odd Schur function can be computed as
\begin{equation*}
\langle s_\lambda,s_\lambda\rangle=(-1)^{\frac{1}{2}s^*(2\lambda)}.
\end{equation*}

%
\section{The even and odd RSK correspondences}\label{sec-rsk}
%

%
\subsection{The classical RSK correspondence}\label{subsec-classical-rsk}
%

The proof of Theorem \ref{thm-RSK-I} is simply a matter of keeping track of some signs in the bijection of the usual RSK correspondence.  Before giving the proof, we briefly review this bijection.  An excellent reference, whose notation we follow, is Chapter 4 of \cite{Fulton}.

Let $A=\lbrace a_1,a_2,\ldots\rbrace$ be an ordered alphabet.  In examples, we will take $A=\Z_{>0}$.  In order to keep track of tableaux systematically, we will use the \textit{plactic monoid} $Pl$, which is the associative monoid (without unit) defined by
\begin{equation}\begin{split}
\text{generators:}\qquad&A\\
\text{relations:}\qquad&yzx=yxz\qquad\text{if }x<y\leq z\qquad(K'),\\
&xzy=zxy\qquad\text{if }x\leq y<z\qquad(K'').
\end{split}\end{equation}
The \textit{plactic ring} $\Z Pl$ is defined to be the $\Z$-span of the plactic monoid, with multiplication induced by that of the monoid.  The relations $(K'),(K'')$ are known as \textit{elementary Knuth transformations}.  We define a map from the set of semistandard Young tableaux with entries in the alphabet $A$ to the plactic monoid or ring by
\begin{equation}\begin{split}
\lbrace\text{SSYTs}\rbrace&\to Pl\text{ or }\Z Pl\\
T&\mapsto w_r(T).
\end{split}\end{equation}
Here, $w_r(T)$ is the row word of $T$, that is, the word obtained by reading the entries of $T$ from left to right, bottom to top (see Subsection \ref{subsec-schur}).  The utility of the plactic ring is in large part due to the following remarkable theorem.
\begin{thm}[Section 2.1 of \cite{Fulton}]\label{thm-word-tableau} Every word is equivalent, via relations $(K')$ and $(K'')$, to the row word $w_r(T)$ of a unique semistandard tableau $T$.\end{thm}
\noindent Thus the set of all semistandard Young tableaux with entries in $A$ forms a basis of $\Z Pl$.  We will informally refer to the multiplication of tableaux in the following; what we mean is the multiplication of their row words in $\Z Pl$.  The relations $(K')$ and $(K'')$ can be interpreted as ``bumping transformations'':
\begin{equation*}\begin{split}
(K')\qquad&\young(yz)\cdot\young(x)=\young(xz,y)\qquad\text{if }x<y\leq z,\\
(K'')\qquad&\young(xz)\cdot\young(y)=\young(xy,z)\qquad\text{if }x\leq y<z.
\end{split}\end{equation*}
For a detailed exposition of bumping, see Section 1.1 of \cite{Fulton}.

\vspace{0.07in}

We remark that if a word $w$ is known to be the row word of some tableau, then it is easy to reconstruct the tableau from the word.  Since the row entries of a tableau never decrease and the column entries must always increase, reading the word $w$ from left to right until the first decrease simply gives the bottom row of the tableau.  Then continuing to read until the next decrease gives the second to bottom row, and so forth.  

\begin{example} Using $\Z_{>0}$ as the ordered alphabet,
\begin{equation*}
w=53422331112\qquad\text{corresponds to}\qquad\young(1112,2233,34,5).
\end{equation*}
\end{example}

Having recalled the language of elementary Knuth transformations and the plactic ring, we proceed to discuss the following result.  Let $N_{\mu\rho}$ be the number of $\N$-matrices $A$ with $\row(A)=\mu$ and $\col(A)=\rho$, and let $n=|\mu|=|\rho|$.
\begin{thm}[RSK Correspondence]\label{thm-RSK} The RSK map is a bijection
\begin{equation}
\text{RSK}:\lbrace\quad
\substack{
\N\text{-matrices }A\text{ with} \\
\row(A)=\mu\text{ and }\col(A)=\rho
}\quad\rbrace\rightarrow\lbrace\quad
\substack{
\text{pairs }(P,Q)\text{ of semistandard Young tableaux of the} \\
\text{same shape, with }\cont(P)=\mu\text{ and }\cont(Q)=\rho}\quad\rbrace.
\end{equation}
In particular,
\begin{equation}\label{eqn-RSK}
N_{\mu\rho}=\sum_{\lambda\vdash n}K_{\lambda\mu}K_{\lambda\rho}.
\end{equation}
\end{thm}

\noindent We now describe the RSK map.  Let $A$ be an $\N$-matrix such that $\row(A)=\mu$ and $\col(A)=\rho$ are partitions and note that the sum of the entries of $A$ equals $n$.  For the purposes of this discussion, consider an entry equal to $k$ to be $k$ distinct entries, each equal to 1, all in the same place.  Order the entries of $A$ from left to right and top to bottom, as if reading a book.  For $j=1,\ldots,n$, let $u_j$ be the row number and $v_j$ the column number of the $j$-th entry in this ordering.  Organize these coordinates into a two-line array,
\begin{equation*}
\binom{\underline{u}}{\underline{v}}=\binom{u_1,u_2,\ldots,u_n}{v_1,v_2,\ldots,v_n}.
\end{equation*}
Note that $\underline{u}$ will always be non-decreasing.

To this two-line array, we associate a pair $(P,Q)$ of tableaux in the following way.  For $k=1,\ldots,n$, let $P_k$ be the unique tableau whose row word $w_r(P_k)$ is Knuth equivalent to $v_1\cdots v_k$, as guaranteed by Theorem \ref{thm-word-tableau}.  This is equivalent to proceeding one box at a time in Schensted's row insertion (bumping) algorithm.  So at each step, $P_k$ is a Young tableau with $k$ boxes whose entries are $\lbrace v_1,\ldots,v_k\rbrace$, and the shape of $P_k$ is obtained from the shape of $P_{k-1}$ by adding one box.  Let $Q_1$ be the one-box Young tableau with entry $u_1$.  Inductively, build $Q_k$ from $Q_{k-1}$ by placing a new box with entry $u_k$ at the location of the box of $P_k$ which was not in $P_{k-1}$.  So at each step, $P_k$ and $Q_k$ have the same shape, and the entries of $Q_k$ are $\lbrace u_1,\ldots,u_k\rbrace$.

The RSK map assigns the pair of final tableaux $(P_n,Q_n)$ to the matrix $A$.

\begin{example}\label{ex-rsk-1}  We illustrate the RSK correspondence and equation \eqref{eqn-RSK} for $\mu=\tiny\yng(1,1,1)$ and $\rho=\tiny\yng(2,1)$.
\begin{align*}
\begin{pmatrix}1&0\\1&0\\0&1\\\end{pmatrix}\qquad&\text{corresponds to}&
\young(112),\quad&\young(123),\\
\begin{pmatrix}1&0\\0&1\\1&0\\\end{pmatrix}\qquad&\text{corresponds to}&
\young(11,2),\quad&\young(12,3),\\
\begin{pmatrix}0&1\\1&0\\1&0\\\end{pmatrix}\qquad&\text{corresponds to}&
\young(11,2),\quad&\young(13,2).
\end{align*}
Indeed, 
\begin{equation*}\begin{split}
&N_{\tiny\yng(1,1,1),\yng(2,1)}=3,\\
&K_{\tiny\yng(3),\yng(1,1,1)}K_{\tiny\yng(3),\yng(2,1)}=1\cdot1,\\
&K_{\tiny\yng(2,1),\yng(1,1,1)}K_{\tiny\yng(2,1),\yng(2,1)}=2\cdot1,\\
&K_{\tiny\yng(1,1,1),\yng(2,1)}K_{\tiny\yng(1,1,1),\yng(1,1,1)}=0\cdot1.
\end{split}\end{equation*}
\end{example}

\begin{example}\label{ex-rsk-2}  We illustrate the RSK correspondence and equation \eqref{eqn-RSK} for $\mu=\tiny\yng(2,2)$ and $\rho=\tiny\yng(2,1,1)$.
\begin{align*}
\begin{pmatrix}2&0&0\\0&1&1\\\end{pmatrix}\qquad&\text{corresponds to}&
\young(1123),\quad&\young(1122),\\
\begin{pmatrix}1&1&0\\1&0&1\\\end{pmatrix}\qquad&\text{corresponds to}&
\young(113,2),\quad&\young(112,2),\\
\begin{pmatrix}1&0&1\\1&1&0\\\end{pmatrix}\qquad&\text{corresponds to}&
\young(112,3),\quad&\young(112,2),\\
\begin{pmatrix}0&1&1\\2&0&0\\\end{pmatrix}\qquad&\text{corresponds to}&\young(11,23),\quad&\young(11,22).
\end{align*}
Indeed, 
\begin{equation*}\begin{split}
&N_{\tiny\yng(2,2),\yng(2,1,1)}=4,\\
&K_{\tiny\yng(4),\yng(2,2)}K_{\tiny\yng(4),\yng(2,1,1)}=1\cdot1,\\
&K_{\tiny\yng(3,1),\yng(2,2)}K_{\tiny\yng(3,1),\yng(2,1,1)}=1\cdot2,\\
&K_{\tiny\yng(2,2),\yng(2,2)}K_{\tiny\yng(2,2),\yng(2,1,1)}=1\cdot1,\\
&K_{\tiny\yng(2,1,1),\yng(2,2)}K_{\tiny\yng(2,1,1),\yng(2,1,1)}=0\cdot1,\\
&K_{\tiny\yng(1,1,1,1),\yng(2,2)}K_{\tiny\yng(1,1,1,1),\yng(2,1,1)}=0\cdot0.\\
\end{split}\end{equation*}
\end{example}

%
\subsection{Proof of the odd RSK correspondence}\label{subsec-odd-rsk}
%

Having reviewed the classical RSK bijection, the proof of the odd RSK Correspondence (originally proved in \cite{Reifegerste}, \cite{Sjostrand}) is simply a matter of keeping track of the signs associated to the combinatorial objects in question.
\begin{proof}[Proof of Theorem \ref{thm-RSK-I}] In passing from a matrix $A$ to the corresponding two-row array $\binom{\underline{u}}{\underline{v}}=\binom{u_1,u_2,\ldots}{v_1,v_2,\ldots}$, we have
\begin{equation*}
\sign(\underline{u})=1,\qquad\sign(\underline{v})=\sign(A).
\end{equation*}
As we construct the semistandard Young tableaux $(P,Q)$ corresponding to $A$ from the words $\underline{u}$ and $\underline{v}$, we will keep track of the signs of their row words at each step.

Each pair $(u_j,v_j)$ describes an entry of the matrix $A$ (we consider an entry equal to 2 as two entries equal to 1, and so forth).  Let $A_1,A_2,\ldots$ be the sequence of matrices obtained by truncating $\underline{u}$ and $\underline{v}$.  That is, under the first step of the RSK correspondence,
\begin{equation*}\begin{split}
&A_1\leftrightarrow\binom{\underline{u}_1}{\underline{v}_1}=\binom{u_1}{v_1},\\
&A_2\leftrightarrow\binom{\underline{u}_2}{\underline{v}_2}=\binom{u_1,u_2}{v_1,v_2},\\
&A_3\leftrightarrow\binom{\underline{u}_3}{\underline{v}_3}=\binom{u_1,u_2,u_3}{v_1,v_2,v_3},
\end{split}\end{equation*}
and so forth.  Let $(P_k,Q_k)$ be the pair of semistandard Young tableaux corresponding to $A_k$, so that $P_k$ (respectively $Q_k$) is obtained from $P_{k-1}$ (respectively $Q_{k-1}$) by adding a box with label $v_k$ (respectively $u_k$).  Let $\lambda_{(k)}=\text{shape}(P_k)=\text{shape}(Q_k)$.  Since the theorem is easily verified for one-box tableaux, it suffices to check that if $A_k\leftrightarrow(P_k,Q_k)$ and $\sign(A_k)=(-1)^{\binom{\lambda_{(k)}^T}{2}}\sign(P_k)\sign(Q_k)$, then passing to $A_{k+1},P_{k+1},Q_{k+1},\lambda_{(k+1)}$ preserves this equality of signs.  (The use of the notation $\lambda_{(k)}$ for a diagram rather than a row length should hopefully cause no confusion.)

To track the sign incurred in adding a box to $P$ and to $Q$, note that both elementary Knuth transformations $(K')$ and $(K'')$ are transpositions (of letters in the words).  It follows that
\begin{equation*}
\sign(\underline{v}_k)=(-1)^{\text{Kn}(\underline{v}_k)}\sign(P_k),
\end{equation*}
where $\text{Kn}(\underline{v}_k)$ is defined to be the number of elementary Knuth transformations needed to re-arrange $\underline{v}_k$ into $w_r(P_k)$.  Its residue modulo 2 is just the sign of the minimal length permutation sorting $\underline{v}_k$ into $w_r(P_k)$.  What needs to be shown, then, is that
\begin{equation}\label{eqn-rsk-proof}\begin{split}
(-1)^{\text{Kn}(\underline{v}_k)}&=(-1)^{\binom{\lambda_{(k)}^T}{2}}\sign(Q_k)\\
&=(-1)^{(\lambda_{(k)})_2+(\lambda_{(k)})_4+(\lambda_{(k)})_6+\ldots}\sign(Q_k).
\end{split}\end{equation}
Since this is clearly true for $k=1$, we can prove this inductively by considering what happens when a new box is added.

Suppose when a new box with label $v_{k+1}$ is added, it ends up in row $s+1$ ($s\geq0$).  In terms of tableaux, this means $s$ boxes were bumped.  Each time a bumping occurs in row $j$, the number of elementary Knuth transformations which take place is $(\lambda_{(k)})_j-1$.  So
\begin{equation*}
\text{Kn}(\underline{v}_{k+1})=\text{Kn}(\underline{v}_k)+\sum_{j=1}^s((\lambda_{(k)})_j-1).
\end{equation*}
And by definition of $s$, the sign $(-1)^{\lambda_2+\lambda_4+\lambda_6+\ldots}$ changes by $(-1)^s$ in passing from $\lambda_{(k)}$ to $\lambda_{(k+1)}$.  Finally,
\begin{equation*}
\sign(Q_{k+1})=(-1)^{(\lambda_{(k)})_1+\ldots+(\lambda_{(k)})_s}\sign(Q_k),
\end{equation*}
since $u_k$ is greater than any label it passes through in sorting $w_r(Q_k)u_{k+1}$ to $w_r(Q_{k+1})$.  We see that the sign changes in the factors of equation \eqref{eqn-rsk-proof} cancel, so the sign equality is preserved under the addition of a new box.  The completes the proof of the theorem.\end{proof}

\begin{example} We return to Example \ref{ex-rsk-1}.  Next to each combinatorial object (including the semistandard Young tableau shapes), we put the associated sign.
\begin{align*}
+\begin{pmatrix}1&0\\1&0\\0&1\\\end{pmatrix}\qquad&+\yng(3)&
+\young(112),\quad&+\young(123),\\
-\begin{pmatrix}1&0\\0&1\\1&0\\\end{pmatrix}\qquad&-\yng(2,1)&
+\young(11,2),\quad&+\young(12,3),\\
+\begin{pmatrix}0&1\\1&0\\1&0\\\end{pmatrix}\qquad&-\yng(2,1)&
+\young(11,2),\quad&-\young(13,2).
\end{align*}
\end{example}

\begin{example}We return to Example \ref{ex-rsk-2}.
\begin{align*}
+\begin{pmatrix}2&0&0\\0&1&1\\\end{pmatrix}\qquad&+\yng(4)&
+\young(1123),\quad&+\young(1122),\\
-\begin{pmatrix}1&1&0\\1&0&1\\\end{pmatrix}\qquad&-\yng(3,1)&
+\young(113,2),\quad&+\young(112,2),\\
+\begin{pmatrix}1&0&1\\1&1&0\\\end{pmatrix}\qquad&-\yng(3,1)&
-\young(112,3),\quad&+\young(112,2),\\
+\begin{pmatrix}0&1&1\\2&0&0\\\end{pmatrix}\qquad&+\yng(2,2)&
+\young(11,23),\quad&+\young(11,22).
\end{align*}
\end{example}

\vspace{0.07in}

We remark that rather than keeping track of the sign of a matrix, then of a two-row array, and finally of two tableaux, we could instead have worked in the \textit{odd plactic ring}, $\Z Pl_{-1}$:
\begin{equation}\begin{split}
\text{generators:}\qquad&A\\
\text{relations:}\qquad&yzx=-yxz\qquad\text{if }x<y\leq z\qquad(K'),\\
&xzy=-zxy\qquad\text{if }x\leq y<z\qquad(K'').
\end{split}\end{equation}
Instead of keeping track of signs associated to various combinatorial objects, then, we could have simply kept track of the sign of certain coefficients in $\Z Pl_{-1}$.

%
\section{Appendix: Data}\label{sec-data}
%

%
\subsection{Bases of $\sym$}\label{subsec-data-bases}
%

In this subsection, we express the monomial, forgotten, and Schur functions in terms of the complete functions for low degrees.

\vspace{0.07in}

Through degree 4, the odd monomial and forgotten functions are:
\begin{align*}
&m_1	=	h_1			&	&m_{1111}	=	h_{1111}-h_{211}+h_{22}-h_4	\\
&m_{11}	=	-h_{11}+h_2	&	&m_{211}	=	-h_{1111}+h_{31}			\\
&m_2	=	h_{11}		&	&m_{22}	=	h_{1111}+h_{22}-2h_4			\\
&m_{111}	=	-h_{111}+h_3	&	&m_{31}	=	h_{211}-h_{31}				\\
&m_{21}	=	-h_{21}+h_3	&	&m_4	=	-h_{1111}-2h_{22}+4h_4		\\
&m_3	=	h_{111}+h_{21}-h_3	&&	\\
&&&\\
&f_1		=	h_1			&	&f_{1111}	=	h_4					\\
&f_{11}	=	h_2			&	&f_{211}	=	h_{31}				\\
&f_2		=	h_{11}		&	&f_{22}	=	-h_{22}+2h_4			\\
&f_{111}	=	h_3			&	&f_{31}	=	h_{211}-h_{31}			\\
&f_{21}	=	-h_{21}+h_3	&	&f_4		=	h_{1111}+2h_{22}-4h_4	\\
&f_3		=	h_{111}+h_{21}-h_3.
\end{align*}

\vspace{0.07in}

Below are odd Kostka numbers through degree 5.  Partitions are listed lexicographically, with shape parametrizing the rows and content parametrizing the columns.

\vspace{0.07in}

\small{
\begin{tabular}{| c | c |}
\hline
deg. 1	&	$(1)$	\\
\hline
$(1)$	&	1		\\
\hline
\end{tabular}

\vspace{0.07in}

\begin{tabular}{| c | c c |}
\hline
deg. 2	&	$(11)$	&	(2)		\\
\hline
$(11)$	&	1		&	0		\\
$(2)$		&	1		&	1		\\
\hline
\end{tabular}

\vspace{0.07in}

\begin{tabular}{| c | c c c |}
\hline
deg. 3	&	$(111)$	&	$(21)$	&	$(3)$		\\
\hline
$(111)$	&	1		&	0		&	0		\\
$(21)$	&	0		&	1		&	0		\\
$(3)$		&	1		&	1		&	1		\\
\hline
\end{tabular}

\vspace{0.07in}

\begin{tabular}{| c | c c c c c |}
\hline
deg. 4	&	$(1111)$	&	$(211)$	&	$(22)$	&	$(31)$	&	$(4)$		\\
\hline
$(1111)$	&	1		&	0		&	0		&	0		&	0		\\
$(211)$	&	1		&	1		&	0		&	0		&	0		\\
$(22)$	&	0		&	1		&	1		&	0		&	0		\\
$(31)$	&	1		&	0		&	$-1$		&	1		&	0		\\
$(4)$		&	1		&	1		&	1		&	1		&	1		\\
\hline
\end{tabular}

\vspace{0.07in}

\begin{tabular}{| c | c c c c c c c |}
\hline
deg. 5	&	$(11111)$	&	$(2111)$	&	$(221)$	&	$(311)$	&	$(32)$	&	$(41)$	&	$(5)$		\\
\hline
$(11111)$	&	1		&	0		&	0		&	0		&	0		&	0		&	0		\\
$(2111)$	&	0		&	1		&	0		&	0		&	0		&	0		&	0		\\
$(221)$	&	$-1$		&	0		&	1		&	0		&	0		&	0		&	0		\\
$(311)$	&	2		&	1		&	$-1$		&	1		&	0		&	0		&	0		\\
$(32)$	&	1		&	1		&	0		&	1		&	1		&	0		&	0		\\
$(41)$	&	0		&	1		&	2		&	0		&	$-1$		&	1		&	0		\\
$(5)$		&	1		&	1		&	1		&	1		&	1		&	1		&	1		\\
\hline
\end{tabular}
}

\vspace{0.07in}

\normalsize{
\noindent Using the above Kostka numbers, we can compute odd Schur functions.  Through degree 5 they are, in the $h$-basis,
\begin{align*}
&s_1		=	h_1						&	&s_{11111}	=	h_{11111}+h_{221}-h_{311}-2h_{41}+h_5	\\
&s_{11}	=	h_{11}-h_2					&	&s_{2111}	=	h_{2111}-h_{311}-h_{41}+h_5			\\
&s_2		=	h_2						&	&s_{221}	=	h_{221}+h_{311}-h_{32}-3h_{41}+2h_5		\\
&s_{111}	=	h_{111}-h_3				&	&s_{311}	=	h_{311}-h_{32}-h_{41}+h_5				\\
&s_{21}	=	h_{21}-h_3					&	&s_{32}	=	h_{32}+h_{41}-2h_5					\\
&s_3		=	h_3						&	&s_{41}	=	h_{41}-h_5							\\
&s_{1111}	=	h_{1111}-h_{211}+h_{22}-h_4	&	&s_5		=	h_5								\\
&s_{211}	=	h_{211}-h_{22}-h_{31}+h_4		&&	\\
&s_{22}	=	h_{22}+h_{31}-2h_4			&&	\\
&s_{31}	=	h_{31}-h_4					&&	\\
&s_4		=	h_4						&&	\\
\end{align*}}

%
\subsection{The bilinear form}\label{subsec-data-bilinear-form}
%

\normalsize{In this subsection, we present some low degree computations regarding the bilinear forms on $\sym'$ and $\sym$.  

\vspace{0.07in}

Let $[n]=1+q+q^2+\ldots+q^{n-1}$ be the (unbalanced) $q$-number and let $[n]!=[n][n-1]\cdots[2][1]$ be the corresponding $q$-factorial.  The bilinear form for unspecialized $q$ is (the form is symmetric, so ``$*$'' stands for the matching entry above the diagonal):}

\vspace{0.07in}

\small{
\begin{tabular}{| c | c |}
\hline
deg. 1	&	$h_1$	\\
\hline
$h_1$	&	1		\\
\hline
\end{tabular}

\vspace{0.07in}

\begin{tabular}{| c | c c |}
\hline
deg. 2	&	$h_{11}$	&	$h_2$	\\
\hline
$h_{11}$	&	[2]		&	1		\\
$h_2$	&	*		&	1		\\
\hline
\end{tabular}

\vspace{0.07in}

\begin{tabular}{| c | c c c c |}
\hline
deg. 3	&	$h_{111}$	&	$h_{12}$	&	$h_{21}$	&	$h_3$	\\
\hline
$h_{111}$	&	$[3]!$	&	$[3]$		&	$[3]$		&	1		\\
$h_{12}$	&	*		&	$[2]$		&	$1+q^2$	&	1		\\
$h_{21}$	&	*		&	*		&	$[2]$		&	1		\\
$h_3$	&	*		&	*		&	*		&	1		\\
\hline
\end{tabular}

\vspace{0.07in}

\begin{tabular}{| c | c c c c c c c c |}
\hline
deg. 4		&	$h_{1111}$	&	$h_{112}$	&	$h_{121}$	&	$h_{211}$	&	$h_{22}$		&	$h_{13}$		&	$h_{31}$		&	$h_4$	\\
\hline
$h_{1111}$	&	$[4]!$	&	$[4][3]$		&	$[4][3]$			&	$[4][3]$		&	$[5]+q^2$		&	$[4]$			&	$[4]$			&	1	\\
$h_{112}$	&	*		&	$[5]+q[2]$		&	$[5]+q^2[2]$		&	$[6]+q^2$		&	$[3]+q^4$		&	$[3]$			&	$[1]+q^2[2]$	&	1	\\
$h_{121}$	&	*		&	*			&	$[4]+q+q^3+q^5$	&	$[5]+q^2[2]$	&	$1+2q^2+q^3$	&	$[2]+q^3$		&	$[2]+q^3$		&	1	\\
$h_{211}$	&	*		&	*			&	*				&	$[5]+q[2]$		&	$[3]+q^4$		&	$1+q^2[2]$	&	$[3]$			&	1	\\
$h_{22}$	&	*		&	*			&	*				&	*			&	$[2]+q^4$		&	$1+q^2$		&	$1+q^2$		&	1	\\
$h_{13}$	&	*		&	*			&	*				&	*			&	*			&	$[2]$			&	$1+q^3$		&	1	\\
$h_{31}$	&	*		&	*			&	*				&	*			&	*			&	*				&	$[2]$			&	1	\\
$h_4$	&	*		&	*		&	*				&	*			&	*			&	*				&	*			&	1	\\
\hline
\end{tabular}
}

\normalsize{With $q=-1$, the bilinear form on the quotient $\sym$ is (through degree 6):}

\vspace{0.07in}

\small{
\begin{tabular}{| c | c |}
\hline
deg. 1	&	$h_1$	\\
\hline
$h_1$	&	1		\\
\hline
\end{tabular}

\vspace{0.07in}

\begin{tabular}{| c | c c |}
\hline
deg. 2	&	$h_{11}$	&	$h_2$	\\
\hline
$h_{11}$	&	0		&	1		\\
$h_2$	&	1		&	1		\\
\hline
\end{tabular}

\vspace{0.07in}

\begin{tabular}{| c | c c c |}
\hline
deg. 3	&	$h_{111}$	&	$h_{21}$	&	$h_3$	\\
\hline
$h_{111}$	&	0		&	1		&	1		\\
$h_{21}$	&	1		&	0		&	1		\\
$h_3$	&	1		&	1		&	1		\\
\hline
\end{tabular}

\vspace{0.07in}

\begin{tabular}{| c | c c c c c |}
\hline
deg. 4		&	$h_{1111}$	&	$h_{211}$	&	$h_{22}$	&	$h_{31}$	&	$h_4$	\\
\hline
$h_{1111}$		&	0		&	0			&	2		&	0		&	1		\\
$h_{211}$		&	0		&	1			&	2		&	1		&	1		\\
$h_{22}$		&	2		&	2			&	1		&	2		&	1		\\
$h_{31}$		&	0		&	1			&	2		&	0		&	1		\\
$h_4$		&	1		&	1			&	1		&	1		&	1		\\
\hline
\end{tabular}

\vspace{0.07in}

\begin{tabular}{| c | c c c c c c c |}
\hline
deg. 5		&	$h_{11111}$	&	$h_{2111}$	&	$h_{221}$	&	$h_{311}$	&	$h_{32}$	&	$h_{41}$	&	$h_5$	\\
\hline
$h_{11111}$	&	0			&	0			&	2			&	0			&	2		&	1		&		1		\\
$h_{2111}$		&	0			&	1			&	0			&	1			&	3		&	0		&		1		\\
$h_{221}$		&	2			&	0			&	$-3$			&	2			&	3		&	$-1$		&		1		\\
$h_{311}$		&	0			&	1			&	2			&	1			&	2		&	1		&		1		\\
$h_{32}$		&	2			&	3			&	3			&	2			&	1		&	2		&		1		\\
$h_{41}$		&	1			&	0			&	$-1$			&	1			&	2		&	0		&		1		\\
$h_5$		&	1			&	1			&	1			&	1			&	1		&	1		&		1		\\
\hline
\end{tabular}

\vspace{0.07in}

\begin{tabular}{| c | c c c c c c c c c c c |}
\hline
deg. 6		&	$h_{111111}$	&	$h_{21111}$	&	$h_{2211}$	&	$h_{222}$	&	$h_{3111}$	&	$h_{321}$	&	$h_{33}$	&	$h_{411}$	&	$h_{42}$	&	$h_{51}$	&	$h_6$	\\
\hline
$h_{111111}$	&	0			&	0			&	0		&	6		&	0		&	0		&	0		&	0		&	3		&	0		&	1		\\
$h_{21111}$	&	0			&	0			&	2		&	6		&	0		&	2		&	2		&	1		&	3		&	1		&	1		\\
$h_{2211}$		&	0			&	2			&	4		&	3		&	2		&	4		&	4		&	2		&	2		&	2		&	1		\\
$h_{222}$		&	6			&	6			&	3		&	$-3$		&	6		&	5		&	5		&	3		&	0		&	3		&	1		\\
$h_{3111}$		&	0			&	0			&	2		&	6		&	0		&	$-1$		&	0		&	1		&	3		&	0		&	1		\\
$h_{321}$		&	0			&	2			&	4		&	5		&	$-1$		&	$-4$		&	$-2$		&	2		&	3		&	$-1$		&	1		\\
$h_{33}$		&	0			&	2			&	4		&	5		&	0		&	$-2$		&	0			&	2		&	3		&	0		&	1		\\
$h_{411}$		&	0			&	1			&	2		&	3		&	1		&	2		&	2		&	1		&	2		&	1		&	1		\\
$h_{42}$		&	3			&	3			&	2		&	0		&	3		&	3		&	3		&	2		&	1		&	2		&	1		\\
$h_{51}$		&	0			&	1			&	2		&	3		&	0		&	$-1$		&	0		&	1		&	2		&	0		&	1		\\
$h_6$		&	1			&	1			&	1		&	1		&	1		&	1		&	1		&	1		&	1		&	1		&	1		\\
\hline
\end{tabular}
}

\vspace{0.07in}

\normalsize{The following are the minimal polynomials for the values of $q$ at which the $q$-bilinear form is degenerate, through degree 7.  (The same values are plotted in Figure 1 of \cite{ThibonUng}.)  Since a nontrivial relation in degree $k$ causes nontrivial relations in all higher degrees, we only list the new minimal polynomials in each degree.  Note that all are monic and with constant coefficient 1, so their roots are units in the ring of algebraic integers.  In fact, most (but not all) are roots of unity.  All of the polynomials are palindromic, so for longer ones we use ellipses ($\ldots$) to denote palindromic continuation.  For instance, $q^5+7q^4-q^3+\ldots$ would mean $q^5+7q^4-q^3-q^2+7q+1$ and $q^4-q^3+2q^2+\ldots$ would mean $q^4-q^3+2q^2-q+1$.  In parentheses, we give the multiplicities in the various degrees in the following format: ((5,1),(6,4),(7,10)) means multiplicity 1 in degree 5, multiplicity 4 in degree 6, and multiplicity 10 in degree 7.

\vspace{0.07in}

\noindent degree 2:
\begin{itemize}
\item $q$ ((2,1),(3,5),(4,17),(5,49),(6,129),(7,321))
\end{itemize}

\vspace{0.07in}

\noindent degree 3:
\begin{itemize}
\item $q-1$ ((3,1),(4,4),(5,14),(6,38),(7,102))
\item $q+1$ ((3,1),(4,4),(5,12),(6,34),(7,88))
\end{itemize}

\vspace{0.07in}

\noindent degree 4:
\begin{itemize}
\item $q^6+2q^4-q^3+2q^2+1$ ((4,1),(5,2),(6,5),(7,12)) (not a root of unity)
\end{itemize}

\vspace{0.07in}

\noindent degree 5:
\begin{itemize}
\item $q^2+q+1$ ((5,2),(6,6),(7,18)) (3rd root of unity)
\item $q^2-q+1$ ((5,1),(6,4),(7,11)) (6th root of unity)
\item $q^{18}+q^{17}+3q^{16}+4q^{15}+6q^{14}+7q^{13}+8q^{12}+10q^{11}+11q^{10}+10q^9+\ldots$ ((5,1),(6,2),(7,5)) (not a root of unity)
\end{itemize}

\vspace{0.07in}

\noindent degree 6:
\begin{itemize}
\item $q^2+1$ ((6,2),(7,8)) (4th root of unity)
\item $q^{10}-q^9+q^8-q^7+q^6-q^5+\ldots$ ((6,1),(7,2)) (22nd root of unity)
\item $q^{50}+q^{49}+2q^{48}+2q^{47}+5q^{46}+4q^{45}+8q^{44}+6q^{43}+11q^{42}+9q^{41}+16q^{40}+16q^{39}+21q^{38}+14q^{37}+23q^{36}+24q^{35}+30q^{34}+23q^{33}+30q^{32}+28q^{31}+38q^{30}+30q^{29}+34q^{28}+30q^{27}+39q^{26}+34q^{25}+\ldots$ ((6,1),(7,2)) (not a root of unity)
\end{itemize}

\vspace{0.07in}

\noindent degree 7:
\begin{itemize}
\item $q^4+q^3+q^2+q+1$ ((7,2)) (5th root of unity)
\item $q^6+q^5+q^4+q^3+q^2+q+1$ ((7,1)) (7th root of unity)
\item $q^4+1$ ((7,1)) (8th root of unity)
\item $q^4-q^3+q^2-q+1$ ((7,1)) (10th root of unity)
\item $q^4-q^2+1$ ((7,1)) (12th root of unity)
\item $q^{16}+q^{15}+q^{14}+q^{13}+q^{12}+q^{11}+q^{10}+q^9+q^8+\ldots$ ((7,1)) (17th root of unity)
\item $q^{12}-q^{10}+q^8-q^6+q^4-q^2+1$ ((7,1)) (28th root of unity)
\item $q^{102}-q^{101}+4q^{100}-2q^{99}+9q^{98}-2q^{97}+18q^{96}-q^{95}+34q^{94}+2q^{93}+58q^{92}+13q^{91}+88q^{90}+36q^{89}+134q^{88}+64q^{87}+204q^{86}+99q^{85}+298q^{84}+155q^{83}+405q^{82}
+238q^{81}+537q^{80}+330q^{79}+705q^{78}+442q^{77}+887q^{76}+584q^{75}+1089q^{74}+731q^{73}+1323q^{72}+881q^{71}+1572q^{70}
+1050q^{69}+1808q^{68}+1233q^{67}+2045q^{66}+1401q^{65}+2284q^{64}+1565q^{63}+2494q^{62}+1716q^{61}+2692q^{60}+1829q^{59}
+2874q^{58}+1926q^{57}+2995q^{56}+2018q^{55}+3067q^{54}+2070q^{53}+3118q^{52}+2080q^{51}+\ldots$ ((7,1)) (not a root of unity)
\end{itemize}

\vspace{0.07in}

In light of the above data, a few remarks on the determinant of the bilinear form are in order.  It is immediate from the definition of the bilinear form that this determinant is monic in $q$.  It is not hard to see that its degree is given by
\begin{equation}\begin{split}\label{eqn-deg-1}
D&=\sum_\alpha\left(\frac{1}{2}n(n-1)-\sum_{i=1}^{\ell(\alpha)}\frac{1}{2}\alpha_i(\alpha_i-1)\right)\\
&=2^{n-2}n(n-1)-\frac{1}{2}\sum_\alpha\sum_{i=1}^{\ell(\alpha)}\alpha_i(\alpha_i-1).
\end{split}\end{equation}
The outer summation is over all compositions $\alpha$ of $n$, and $\ell(\alpha)$ is the length of the composition $\alpha$.  Let
\begin{equation*}
A_n=\sum_\alpha\sum_{i=1}^{\ell(\alpha)}\alpha_i(\alpha_i-1)
\end{equation*}
be the double summation of equation \eqref{eqn-deg-1}.  Re-indexing so as to first sum over the first entry of each composition, we see
\begin{equation*}
A_n=n(n-1)+\sum_{k=1}^{n-1}\left[2^{n-k-1}k(k-1)+A_{n-k}\right].
\end{equation*}
This recursion can be solved as
\begin{equation*}
A_n=2+2^n(n-2),
\end{equation*}
from which we conclude that the degree of the determinant of the bilinear form, as a polynomial in $q$, equals
\begin{equation}\label{eqn-deg-2}
D=2^{n-2}\left(n^2-3n+4\right)-1.
\end{equation}}

\addcontentsline{toc}{section}{References}


\bibliographystyle{plain}
\bibliography{ellis-bib}

%

\vspace{0.1in}

\noindent A.P.E.: { \sl \small Department of Mathematics, Columbia University, New
York, NY 10027, USA} \newline \noindent {\tt \small email: ellis@math.columbia.edu}

\vspace{0.1in}

\noindent M.K.: { \sl \small Department of Mathematics, Columbia University, New
York, NY 10027, USA} \newline \noindent {\tt \small email: khovanov@math.columbia.edu}

%
\end{document}